\documentclass[11pt,a4paper]{article}

\usepackage[utf8]{inputenc}
\usepackage[T1]{fontenc}
\usepackage{amsmath,amssymb,amsthm}
\usepackage{mathtools}
\usepackage{mathrsfs}
\usepackage{graphicx}
\usepackage{hyperref}
\usepackage{geometry}
\usepackage{enumitem}
\usepackage{float}
\usepackage{multirow}
\usepackage{array}
\usepackage{booktabs}

\usepackage{pgfplots}
\pgfplotsset{compat=1.16}
\usepackage{tikz}
\usetikzlibrary{arrows.meta}
\hypersetup{
    colorlinks=true,
    linkcolor=blue,
    citecolor=blue,
    urlcolor=blue
}

\newtheorem{theorem}{Theorem}[section]
\newtheorem{lemma}[theorem]{Lemma}
\newtheorem{proposition}[theorem]{Proposition}
\newtheorem{corollary}[theorem]{Corollary}
\newtheorem{definition}[theorem]{Definition}
\newtheorem{remark}[theorem]{Remark}

\newcommand{\R}{\mathbb{R}}

\DeclareMathOperator{\atanTwo}{atan2}

\title{\bfseries The Return Map in the Class $\mathcal{O}_C$ : \\
       Geometry, Dynamics, and Thickness Regularity }
\author{
	M. Barkatou\thanks{ISTM Laboratory, Chouaib Doukkali University, Morocco. \texttt{barkatou.m@ucd.ac.ma}}\\
	M. El Morsalani\thanks{QWave Consult, Germany. \texttt{Mohamed.elmorsalani@qwave-consult.eu}}}
\date{}
\begin{document}
\maketitle

\begin{abstract}
We study a geometrically defined return map associated with domains
in the class $\mathcal O_C$, where a fixed convex core
$C\subset\mathbb R^N$ is linked to the outer boundary
$\partial\Omega$ through outward normal rays and a thickness
function $d:\partial C\to(0,\infty)$. The radial map
$\Phi(c)=c+d(c)\nu(c)$ is combined with a reciprocal map that
follows the inward normal to $\partial\Omega$ back to its first
intersection with the convex core. Their composition
$F=\pi\circ\Phi:\partial C\to\partial C$ defines a geometrically generated discrete dynamical system on
$\partial C$.
Using the exact normal geometry of the radial parametrization, we
derive the local quadratic-remainder formula
$F(c)-c=d(c)\bigl(I-d(c)S_c\bigr)^{-1}\nabla_{\partial C}d(c)
+O\!\left(|\nabla_{\partial C}d(c)|^2\right)$,
where $S_c$ is the shape operator of $\partial C$. Thus curvature
enters the leading-order dynamics through a positive anisotropic
preconditioning of the thickness gradient. At points where the relevant derivatives exist, fixed points			of the return map coincide with critical points
of the thickness function. Under additional smoothness assumptions,
we obtain the linearization
$DF(c_*)=I+d_*\bigl(I-d_*S_*\bigr)^{-1}H_*$,
with $H_*$ the Hessian operator of $d$ at $c_*$. This yields a local
stability classification governed jointly by thickness variation
and curvature.
The analytical results are validated by exact unit-circle
computations. We also establish regularity transfer between the
thickness function and the outer boundary and derive global
bilipschitz bounds for the radial parametrization from convex metric
projection. These results provide a geometric foundation for the
study of boundary-induced discrete dynamics generated by normal
interactions.
\end{abstract}
\noindent\textbf{Keywords:} return map, geometric dynamics, thickness function,  convex geometry, normal parametrization,  curvature-preconditioned dynamics, discrete dynamical systems, bilipschitz maps.

\noindent\textbf{2020 Mathematics Subject Classification:}
Primary 37C05; Secondary 37C10, 53A05, 53C21, 49Q10.
\tableofcontents

\section{Introduction}

Geometric properties of domain boundaries play a central role in
convex geometry, shape analysis, and geometric dynamical systems.
Among the basic structures involved are normal parametrizations,
metric projections, curvature operators, and the regularity of
offset or parallel boundaries. These ideas also arise naturally in
shape optimization, where geometric information carried by the
boundary is used to describe admissible variations and associated
dynamical processes; see, for example,
\cite{1,2,3,4,5}.

A related point of view arises in the geometric-normal framework
introduced by Barkatou \cite{6}. In this setting, a fixed convex set
$C\subset\mathbb R^N$ is contained in a larger domain $\Omega$, and
the geometry of the two boundaries is linked through normal rays.
For the subclass $\mathcal O_C$ considered here, every point
$c\in\partial C$ determines a unique point of the outer boundary
along the outward normal to the convex core. This gives the radial
representation
\[
\Phi(c)=c+d(c)\nu(c),
\]
where $\nu(c)$ is the outward unit normal to $\partial C$ and
$d:\partial C\to(0,\infty)$ is the associated thickness function.

Normal parametrizations of convex geometries and their curvature
properties continue to play an important role in current geometric
analysis. Recent work includes normal-coordinate descriptions of
strictly convex bodies and offset geometries \cite{7}, structural
results for sets of positive reach \cite{8}, and generalized
curvature measures associated with normal geometry \cite{9}.
Related developments in shape analysis emphasize the differential
geometry of admissible domains and shape spaces \cite{4,5}.
Discrete gradient-type constructions have also been studied from
the viewpoint of geometric integration \cite{10}. These works
provide a natural contemporary setting for the interaction between
normal geometry, curvature, and discrete evolution considered in
the present paper.

A central motivation for studying this construction is that it provides a direct passage from static boundary geometry to a geometrically generated discrete dynamics. The return map is not imposed as an
external numerical scheme, optimization algorithm, or prescribed
flow. It is generated by the geometry itself: the two boundary
shapes, their normal fields, and the thickness separating them
determine the iteration. Consequently, three objects that are
usually studied separately become linked within a single
construction: the geometry of the admissible domain, the scalar
thickness landscape on $\partial C$, and a discrete dynamical
system on the convex-core boundary.

This interaction is already visible at first order. The thickness
gradient does not act alone; it is transformed by the
curvature-dependent operator
\[
d(c)\bigl(I-d(c)S_c\bigr)^{-1}.
\]
Thus local motion records simultaneously how the thickness varies
and how the underlying convex boundary bends. This provides a
natural mechanism by which geometric information is converted into
dynamical information. More broadly, it raises the possibility of
studying normal interactions between nested or constrained
boundaries through dynamical methods, and conversely of extracting
geometric information from the resulting dynamics. Establishing
this mechanism rigorously is therefore a natural first step before
addressing its global dynamical, variational, or topological
consequences.

The mechanism studied here is generated directly by the geometry
of the two boundaries. Starting from $c\in\partial C$, one first
moves along the outward normal of the convex core to
\[
x=\Phi(c)\in\partial\Omega\setminus C.
\]
From $x$, one then follows the inward normal to
$\partial\Omega$ until it meets the convex core for the first time.
If
\[
\pi:\partial\Omega\setminus C\longrightarrow\partial C
\]
denotes this reciprocal map, the composition
\[
F=\pi\circ\Phi:\partial C\longrightarrow\partial C
\]
defines a discrete return map. Thus one iteration is the geometric
round-trip
\[
c
\longmapsto
\Phi(c)
\longmapsto
F(c).
\]
The two legs are governed by different normal fields: the first by
the outward normal to $\partial C$, and the second by the inward
normal to $\partial\Omega$. Their interaction is the source of the
nontrivial return dynamics.

The first objective of the paper is to derive the local structure
of this map directly from the geometry. With the convention
\[
S_c=-D\nu(c)\big|_{T_c\partial C}
\]
for the shape operator of the convex core, the differential of the
radial map naturally introduces the positive tangential operator
\[
I-d(c)S_c.
\]
The resulting first-order expansion of the return map near a
critical point of the thickness is
\[
F(c)-c
=
d(c)\bigl(I-d(c)S_c\bigr)^{-1}
\nabla_{\partial C}d(c)
+
O\!\left(
|\nabla_{\partial C}d(c)|^2
\right).
\]
The leading term is therefore not an ordinary gradient step.
Instead, the tangential gradient of the thickness is modified by
the curvature-dependent operator
\[
B_c
=
d(c)\bigl(I-d(c)S_c\bigr)^{-1}.
\]
In principal directions of the convex core, its eigenvalues are
\[
\frac{d(c)}
{1+d(c)\kappa_i(c)},
\]
so the curvature of $\partial C$ enters explicitly and
anisotropically into the local displacement.

A second objective is to identify the equilibrium structure
associated with this geometry. We show that, at points where the
relevant derivatives exist,
\[
F(c)=c
\quad\Longleftrightarrow\quad
\nabla_{\partial C}d(c)=0.
\]
Under stronger local smoothness assumptions, the derivative of the
return map at a critical point $c_*$ is
\[
DF(c_*)
=
I+B_*H_*,
\]
where
\[
B_*
=
d_*\bigl(I-d_*S_*\bigr)^{-1}
\]
and $H_*$ is the Hessian operator of the thickness function at
$c_*$. Since $B_*$ is positive definite, the spectrum of the
linearization can be expressed through the self-adjoint operator
\[
B_*^{1/2}H_*B_*^{1/2}.
\]
This yields a local stability classification in which the
second-order geometry of the thickness and the curvature of the
convex core enter simultaneously.

The analytical formulas are complemented by exact computations for
a unit-circle core. In that case the geometric return can be
evaluated explicitly, allowing the first-order displacement and
the local multiplier to be checked directly from the normal
construction rather than from the asymptotic formula itself. These
examples are used as numerical validation of the local theory.

The radial representation also has consequences independently of
the dynamics. We establish regularity transfer between the
thickness function and the outer boundary under the appropriate
transversality assumptions. Moreover, since the inverse radial map
is the restriction of the metric projection onto the closed convex
set $C$, its non-expansiveness gives the global lower estimate
\[
|c_1-c_2|
\leq
|\Phi(c_1)-\Phi(c_2)|.
\]
Together with the Lipschitz regularity of $d$ and of the normal
field, this yields global bilipschitz control of the radial
parametrization without a perturbative smallness condition.

The purpose of the present work is to establish these geometric and
local dynamical foundations. Questions concerning global orbit
structure, Lyapunov or monotonicity properties, nontrivial periodic
behaviour, continuous-time limits, and topological constraints on
the equilibrium structure require additional hypotheses and are
left for separate investigation.

The paper is organized as follows. Section~2 introduces the
geometric setting, the convex core, and the class $\mathcal O_C$,
and establishes radial accessibility. Section~3 defines the
thickness function and the radial parametrization. Section~4
introduces the reciprocal map and the geometric round-trip, while
Section~5 formulates the induced discrete return dynamics and its
exact closure relations. Section~6 derives the first-order
expansion and discusses its geometric interpretation, including
circle and sphere consistency checks. Section~7 characterizes the
fixed points, derives the linearization, and studies local
stability. Section~8 provides numerical validation using the exact
unit-circle return map. Section~9 establishes the regularity and
bilipschitz properties of the radial parametrization.
Section~10 discusses the scope of the construction and directions
for further study, and Section~11 concludes the paper.

\section{Geometric Setting and the Class $\mathcal{O}_C$}\label{sec:geom}

\subsection{Ambient Space and Notation}

Let $N\geq 2$ and let $\R^N$ be endowed with the Euclidean inner
product $\langle\cdot,\cdot\rangle$ and norm $|\cdot|$. For a set
$A\subset\R^N$, we denote by $\partial A$ and $\operatorname{int}(A)$
its boundary and interior, respectively.

Throughout the paper, $C\subset\R^N$ denotes the fixed convex core.
At a regular point $c\in\partial C$, the symbol $\nu(c)$ denotes the
\emph{outward} unit normal to $\partial C$. For
$x\in\partial\Omega\setminus C$, the symbol $n(x)$ denotes the
\emph{inward} unit normal to $\partial\Omega$ whenever it exists.
Thus the two normal fields have opposite geometric roles and will not
be denoted by the same symbol.

The metric projection onto the closed convex set $C$ is denoted by
\[
P_C:\R^N\longrightarrow C.
\]
For $x\notin C$, the projection $P_C(x)$ is unique.

If $\partial C$ is differentiable at $c$, we use the sign convention
\[
S_c=-D\nu(c)\big|_{T_c\partial C}
\]
for the shape operator; see, e.g., \cite{11}. Hence, for a convex hypersurface with outward
normal, the eigenvalues of $S_c$ are $-\kappa_i(c)$, where
$\kappa_i(c)\geq0$ are the usual principal curvatures. In particular,
\[
I-d(c)S_c
\]
has eigenvalues $1+d(c)\kappa_i(c)>0$ whenever $d(c)\geq0$.

Tangential gradient and Hessian on $\partial C$ are denoted by
\[
\nabla_{\partial C}d
\qquad\text{and}\qquad
\operatorname{Hess}_{\partial C}d,
\]
respectively.

\subsection{The Convex Core $C$}

Let $C\subset\R^N$ be a compact convex set with nonempty interior.
Throughout the dynamical part of the paper we assume that
$\partial C$ is of class $C^{1,1}$. Consequently, the outward unit
normal field
\[
\nu:\partial C\longrightarrow \mathbb S^{N-1}
\]
is well defined and Lipschitz continuous.

For $c\in\partial C$, let
\[
N_C(c)
=
\left\{
\xi\in\R^N:
\langle \xi,z-c\rangle\leq 0
\quad\text{for every }z\in C
\right\}
\]
denote the normal cone of $C$ at $c$. This is the vector form of the
normal cone used in the geometric-normal framework of Barkatou
\cite{6}. Since $\partial C$ is $C^1$, the normal cone is
one-dimensional:
\[
N_C(c)=\{r\nu(c):r\geq0\}.
\]

Convexity also implies that the metric projection
\[
P_C:\R^N\longrightarrow C
\]
is uniquely defined and non-expansive; see, e.g., \cite{12}. For $x\notin C$ and $c\in\partial C$,
\[
c=P_C(x)
\quad\Longleftrightarrow\quad
x-c\in N_C(c).
\]
Hence, under the present regularity assumption,
\begin{equation}
c=P_C(x)
\Longleftrightarrow
x=c+r\nu(c)
\quad\text{for some }r>0.
\label{eq:metric-projection-normal-ray}
\end{equation}

Thus every point outside $C$ lies on a unique outward normal ray
issued from its metric projection onto $\partial C$.

Since $\partial C$ is $C^{1,1}$, the derivative $D\nu$ exists almost
everywhere on $\partial C$ and is essentially bounded. With the sign
convention fixed in Section~2.1,
\[
S_c=-D\nu(c)\big|_{T_c\partial C}
\]
is therefore defined almost everywhere. If
$\kappa_1(c),\ldots,\kappa_{N-1}(c)\geq0$ denote the usual principal
curvatures of the convex hypersurface at such a point, then the
eigenvalues of $S_c$ are
\[
-\kappa_1(c),\ldots,-\kappa_{N-1}(c).
\]
\subsection{The Geometric Normal Property}

We recall the geometric normal property introduced by Barkatou
\cite{6}. Let
\[
\mathcal{N}_{\Omega}
=
\left\{
x\in\partial\Omega :
n(x)\ \text{exists}
\right\}
\]
denote the set of regular boundary points of $\Omega$, where $n(x)$
is the inward unit normal introduced in Section~2.1.

Since $\partial C$ is of class $C^{1,1}$, the selected normal of
Barkatou's general formulation is uniquely determined at every
$c\in\partial C$. We therefore define the outward half-normal
\begin{equation}
\Delta_c^+
=
\{c+r\nu(c):r\geq0\}.
\label{eq:selected-normal-ray}
\end{equation}
and
$$
\ell_c
=
\{c+r\nu(c):r\in\mathbb R\},
$$
as the selected normal line.
\begin{definition}[$C$-geometric normal property]
	An open set $\Omega\subset\R^N$ is said to satisfy the
	$C$-geometric normal property, abbreviated $C$-GNP, if the following
	conditions hold:
	\begin{enumerate}
		\item[(P1)] $\operatorname{int}(C)\subset\Omega$;
		
		\item[(P2)] $\partial\Omega$ is Lipschitz outside $C$;
		
		\item[(P3)] for every $x\in\mathcal{N}_{\Omega}\setminus C$,
		the inward normal half-line
		\[
		\{x+t\,n(x):t\geq0\}
		\]
		intersects $C$;
		
		\item[(P4)] for every $c\in\partial C$, the set
		$\ell_c\cap\Omega$ is connected.
	\end{enumerate}
\end{definition}

The formulation above is the specialization of the original
definition in \cite{6} to a convex core with $C^{1,1}$ boundary.
In particular, the $C$-GNP is a geometric condition and does not by
itself impose the differentiability required later for the return-map
analysis.

We shall use the following consequence of Barkatou's geometric
theory \cite{6}. If $\Omega$ satisfies the $C$-GNP, then every outward
half-normal $\Delta_c^+$ intersects $\partial\Omega$ in at most one
point. Equivalently,
\begin{equation}
\#\bigl(\Delta_c^+\cap\partial\Omega\bigr)\leq 1,
\qquad c\in\partial C.
\label{eq:normal-ray-uniqueness}
\end{equation}
This follows from the separation characterization of the $C$-GNP
and the associated uniqueness result for half-lines contained in the
normal cone \cite{6}.
\subsection{The Class $\mathcal{O}_C$}

We now specify the subclass of $C$-GNP domains on which the return
dynamics will be defined. The additional assumptions below are not
part of the original $C$-GNP of Barkatou; they provide the regularity
and radial accessibility required for the differential analysis of
the return map.

\begin{definition}
	Let $C\subset\R^N$ be the compact convex core introduced above.
	The class $\mathcal{O}_C$ consists of open sets
	$\Omega\subset\R^N$ satisfying the following conditions:
	\begin{enumerate}
		\item[(O1)] $\Omega$ satisfies the $C$-geometric normal property;
		
		\item[(O2)] $\partial\Omega\setminus C$ is an orientable
		hypersurface of class $C^{1,1}$;
		
		\item[(O3)] for every $c\in\partial C$, the outward normal
		half-line $\Delta_c^+$ defined in
		\eqref{eq:selected-normal-ray} intersects
		$\partial\Omega\setminus C$;
		
		\item[(O4)] this intersection is transverse.
	\end{enumerate}
\end{definition}

By \eqref{eq:normal-ray-uniqueness}, condition \textnormal{(O3)}
determines a unique point of $\partial\Omega$ on each outward normal
ray. We denote this point by
\begin{equation}
\Phi(c)=c+d(c)\nu(c),
\qquad c\in\partial C,
\label{eq:radial-point}
\end{equation}
where $d(c)>0$ is the corresponding radial distance.

The transversality condition in \textnormal{(O4)} means that the
normal ray is not tangent to $\partial\Omega$ at $\Phi(c)$. At points
where the inward unit normal $n(\Phi(c))$ is defined, this is
equivalent to
\begin{equation}
\bigl\langle
n(\Phi(c)),\nu(c)
\bigr\rangle\neq0.
\label{eq:radial-transversality}
\end{equation}

Thus $\mathcal{O}_C$ is a regular subclass of the $C$-GNP geometry
for which every point of $\partial C$ admits a well-defined outward
radial link to the outer boundary.
\subsection{Radial Accessibility}

The preceding assumptions imply that the outer boundary is globally
accessible from $\partial C$ along the outward normal rays. The key
point is that this statement follows from convex projection and does
not require the inward normal from $\partial\Omega$ to terminate at
the metric projection.

\begin{proposition}[Radial accessibility]
	\label{prop:radial-accessibility}
	Let $\Omega\in\mathcal{O}_C$. Then for every
	$x\in\partial\Omega\setminus C$ there exists a unique
	$c\in\partial C$ and a unique $r>0$ such that
	\begin{equation}
	x=c+r\nu(c).
	\label{eq:radial-representation}
	\end{equation}
	Moreover,
	\begin{equation}
	c=P_C(x).
	\label{eq:radial-projection}
	\end{equation}
	Consequently, the radial map
	\[
	\Phi:\partial C\longrightarrow\partial\Omega\setminus C
	\]
	defined in \eqref{eq:radial-point} is bijective and satisfies
	\begin{equation}
	\Phi^{-1}
	=
	P_C\big|_{\partial\Omega\setminus C}.
	\label{eq:radial-inverse-projection}
	\end{equation}
\end{proposition}

\begin{proof}
	Let $x\in\partial\Omega\setminus C$ and set
	\[
	c=P_C(x).
	\]
	Since $C$ is closed and convex, the metric projection is uniquely
	defined. Because $x\notin C$, one has $c\in\partial C$. By
	\eqref{eq:metric-projection-normal-ray},
	\[
	x=c+r\nu(c)
	\]
	for some $r>0$. Thus $x$ belongs to the outward normal half-line
	$\Delta_c^+$.
	
	By condition \textnormal{(O3)}, $\Delta_c^+$ meets
	$\partial\Omega\setminus C$, while
	\eqref{eq:normal-ray-uniqueness} shows that it can meet
	$\partial\Omega$ in at most one point. Hence this point must be
	$x=\Phi(c)$.
	
	The uniqueness of $c$ follows from the uniqueness of the metric
	projection onto the convex set $C$, and then $r=|x-c|$ is also
	unique. Therefore $\Phi$ is bijective, and
	\eqref{eq:radial-inverse-projection} follows.
\end{proof}
\begin{remark}
	\label{rem:projection-lipschitz}
	Since the metric projection onto a closed convex set is
	non-expansive, Proposition~\ref{prop:radial-accessibility} yields
	\begin{equation}
	|c_1-c_2|
	\leq
	|\Phi(c_1)-\Phi(c_2)|
	\qquad
	\text{for all }c_1,c_2\in\partial C.
	\label{eq:radial-lower-lipschitz}
	\end{equation}
	This estimate will be used in Section~9 to obtain the lower
	bilipschitz bound without a smallness assumption.
\end{remark}
\section{Thickness Function and Radial Parametrization}

\subsection{Definition of the Thickness Function}

Let $\Omega\in\mathcal{O}_C$. By
Proposition~\ref{prop:radial-accessibility}, for every
$c\in\partial C$ there exists a unique point of
$\partial\Omega\setminus C$ on the outward normal ray issued from
$c$. We denote the corresponding distance by
\[
d(c)>0.
\]

Equivalently, $d:\partial C\to(0,\infty)$ is characterized by
\begin{equation}
\Phi(c)
=
c+d(c)\nu(c)
\in
\partial\Omega\setminus C,
\qquad c\in\partial C.
\label{eq:thickness-definition}
\end{equation}

We call $d$ the \emph{thickness function} of $\Omega$ relative to
the convex core $C$.

Since
\[
c=P_C(\Phi(c))
\]
by Proposition~\ref{prop:radial-accessibility}, the thickness also
admits the geometric representation
\begin{equation}
d(c)
=
|\Phi(c)-c|
=
\operatorname{dist}\bigl(\Phi(c),C\bigr).
\label{eq:thickness-distance}
\end{equation}

Thus $d(c)$ measures the distance from the convex core to the outer
boundary along the unique outward normal ray based at $c$.
\subsection{Radial Parametrization}

The thickness function defines the radial map
\begin{equation}
\Phi:\partial C\longrightarrow\partial\Omega\setminus C,
\qquad
\Phi(c)=c+d(c)\nu(c).
\label{eq:radial-map}
\end{equation}

By Proposition~\ref{prop:radial-accessibility}, $\Phi$ is bijective
and its inverse is the restriction of the metric projection onto the
convex core:
\begin{equation}
\Phi^{-1}
=
P_C\big|_{\partial\Omega\setminus C}.
\label{eq:radial-map-inverse}
\end{equation}

Consequently, the outer boundary admits the global normal
representation
\begin{equation}
\partial\Omega\setminus C
=
\left\{
c+d(c)\nu(c):
c\in\partial C
\right\}.
\label{eq:outer-boundary-normal-graph}
\end{equation}

Equivalently, for every $x\in\partial\Omega\setminus C$,
\begin{equation}
x
=
\Phi(P_C(x))
=
P_C(x)
+
d(P_C(x))\,\nu(P_C(x)).
\label{eq:radial-reconstruction}
\end{equation}

The identities
\begin{equation}
P_C\circ\Phi
=
\operatorname{Id}_{\partial C},
\qquad
\Phi\circ P_C
=
\operatorname{Id}_{\partial\Omega\setminus C}
\label{eq:radial-projection-identities}
\end{equation}
will be used repeatedly below.
\begin{figure}[t]
	\centering
	\resizebox{0.92\columnwidth}{!}{%
		\begin{tikzpicture}[
		>=Latex,
		line cap=round,
		line join=round,
		every node/.style={font=\small}
		]
		
		\draw[thick] (0.4,0.1) ellipse [x radius=3.6, y radius=2.0];
		
		\draw[thick] (0,0) ellipse [x radius=1.75, y radius=1.0];
		
		\node at (-0.9,-1.35) {$\partial C$};
		\node at (2.15,-1.95) {$\partial\Omega\setminus C$};
		
		\coordinate (c) at (1.45,0.48);
		\coordinate (x) at (3.00,1.02);
		
		\fill (c) circle (1.5pt);
		\fill (x) circle (1.5pt);
		
		\node[below right] at (c) {$c$};
		\node[above right] at (x) {$x=\Phi(c)$};
		
		\draw[->,thick] (c) -- (x)
		node[pos=0.52, above=5pt] {$\Phi$};
		
		\node at (0.1,0.42) {$x-c=d(c)\nu(c)$};
		
		\draw[->,dotted,thick] (x) to[bend right=30]
		node[pos=0.52, below=12pt] {$P_C$} (c);
		
		\draw[->] (c) -- ++(0.88,0.31)
		node[right] {$\nu(c)$};
		
		\end{tikzpicture}%
	}
\end{figure}
	
\subsection{Regularity of the Thickness Function}

We record here only the regularity needed to formulate the
differential geometry of the radial map. Higher-order regularity is
considered separately in Section~9.

\begin{proposition}
	\label{prop:thickness-local-lipschitz}
	Let $\Omega\in\mathcal{O}_C$. Then the thickness function
	\[
	d:\partial C\longrightarrow(0,\infty)
	\]
	is locally Lipschitz.
\end{proposition}

\begin{proof}
	Fix $c_0\in\partial C$ and set
	\[
	x_0=\Phi(c_0)\in\partial\Omega\setminus C.
	\]
	Since $\partial\Omega\setminus C$ is of class $C^{1,1}$, it admits
	in a neighbourhood of $x_0$ a $C^{1,1}$ defining function $\rho$,
	with
	\[
	\partial\Omega=\{\rho=0\},
	\qquad
	\nabla\rho(x_0)\neq0.
	\]
	
	Choose a local $C^{1,1}$ parametrization
	$c=c(y)$ of $\partial C$ near $c_0$ and consider
	\[
	G(y,t)
	=
	\rho\bigl(c(y)+t\nu(c(y))\bigr).
	\]
	The map $G$ is locally Lipschitz in $y$. By the transversality
	condition \eqref{eq:radial-transversality}, the derivative in the
	radial variable does not vanish at the intersection:
	\[
	\partial_tG(y_0,d(c_0))
	=
	\left\langle
	\nabla\rho(x_0),\nu(c_0)
	\right\rangle
	\neq0.
	\]
	After restricting the neighbourhood if necessary, this derivative
	is bounded away from zero.
	Thus there exists $\eta>0$ such that
	\[
	|\partial_tG(y,t)|\geq\eta
	\]
	throughout a sufficiently small neighbourhood of
	$(y_0,d(c_0))$. Moreover, $G$ is uniformly Lipschitz in the
	$y$-variable there; let $L$ be a corresponding Lipschitz
	constant. Hence, for nearby $y_1,y_2$,
	\[
	\eta
	\left|
	d(c(y_1))-d(c(y_2))
	\right|
	\leq
	L|y_1-y_2|.
	\]
	Therefore $d$ is locally Lipschitz.
\end{proof}

At points where $d$ and the normal field $\nu$ are differentiable,
the differential of the radial parametrization
\eqref{eq:radial-map} is obtained directly from the product rule.
For $v\in T_c\partial C$,
\begin{equation}
D\Phi(c)[v]
=
v
+
\left\langle
\nabla_{\partial C}d(c),v
\right\rangle\nu(c)
+
d(c)D\nu(c)[v].
\label{eq:differential-radial-map-pre}
\end{equation}
Using the convention
\[
S_c=-D\nu(c)\big|_{T_c\partial C},
\]
this becomes
\begin{equation}
D\Phi(c)[v]
=
\bigl(I-d(c)S_c\bigr)v
+
\left\langle
\nabla_{\partial C}d(c),v
\right\rangle\nu(c).
\label{eq:differential-radial-map}
\end{equation}

It is convenient to introduce the tangential operator
\begin{equation}
L_c:=I-d(c)S_c.
\label{eq:Lc-definition}
\end{equation}
At almost every $c\in\partial C$, its eigenvalues are
\[
1+d(c)\kappa_i(c),
\qquad i=1,\ldots,N-1,
\]
and are therefore strictly positive. In particular, $L_c$ is
invertible almost everywhere.

The differential formula
\eqref{eq:differential-radial-map} will be the starting point for
the computation of the normal to $\partial\Omega$ in Section~6.
\section{Reciprocal Map and Geometric Round-Trip}

\subsection{Normal Rays from $\partial\Omega$}

Let $\Omega\in\mathcal{O}_C$ and let
$x\in\partial\Omega\setminus C$. Since
$\partial\Omega\setminus C$ is an orientable $C^{1,1}$
hypersurface, its inward unit normal $n(x)$ is well defined. We
consider the inward normal half-line
\begin{equation}
\Gamma_x
=
\{x+t\,n(x):t\geq0\}.
\label{eq:inward-normal-ray}
\end{equation}

By the $C$-geometric normal property,
\[
\Gamma_x\cap C\neq\varnothing.
\]
Since $C$ is compact and convex, the set of parameters
\[
I_x
=
\{t\geq0:x+t\,n(x)\in C\}
\]
is a nonempty compact interval. Moreover, $x\notin C$, so that
\[
\inf I_x>0.
\]

We therefore define the first hitting time
\begin{equation}
t(x)
=
\min\{t>0:x+t\,n(x)\in C\}.
\label{eq:first-hitting-time}
\end{equation}
The corresponding point
\begin{equation}
x+t(x)n(x)
\label{eq:first-inward-hit}
\end{equation}
belongs to $\partial C$.

Thus every inward normal ray issued from
$\partial\Omega\setminus C$ possesses a unique first intersection
with the convex core.
\subsection{Reciprocal Map}

The first intersection of the inward normal ray with the convex core
defines the reciprocal map
\begin{equation}
\pi:\partial\Omega\setminus C\longrightarrow\partial C,
\qquad
\pi(x)=x+t(x)n(x),
\label{eq:reciprocal-map}
\end{equation}
where $t(x)>0$ is the first hitting time introduced in
\eqref{eq:first-hitting-time}.

Equivalently,
\begin{equation}
t(x)
=
|\pi(x)-x|.
\label{eq:reciprocal-distance}
\end{equation}

The reciprocal map should be distinguished from the metric projection
onto $C$. In general,
\begin{equation}
\pi(x)\neq P_C(x),
\label{eq:reciprocal-vs-projection}
\end{equation}
since $\pi(x)$ is determined by the inward normal to
$\partial\Omega$ at $x$, whereas $P_C(x)$ is determined by the
outward normal geometry of the convex core.

Equality holds precisely when the inward normal ray at $x$ passes
through the metric projection $P_C(x)$.
\subsection{Basic Properties}

The reciprocal map is well defined on the whole outer boundary by
the $C$-geometric normal property. Its regularity, however, depends
on the manner in which the inward normal ray meets $\partial C$.

\begin{definition}
	\label{def:return-transverse}
	A point $x\in\partial\Omega\setminus C$ is called
	\emph{return-transverse} if
	\begin{equation}
	\left\langle
	n(x),\nu(\pi(x))
	\right\rangle\neq0.
	\label{eq:return-transversality}
	\end{equation}
	Since $\pi(x)$ is the first point at which the inward ray meets the
	convex core, a transverse intersection satisfies in fact
	\[
	\left\langle
	n(x),\nu(\pi(x))
	\right\rangle<0.
	\]
\end{definition}

\begin{proposition}
	\label{prop:reciprocal-basic-properties}
	Let $\Omega\in\mathcal{O}_C$.
	
	\begin{enumerate}
		\item The reciprocal map
		\[
		\pi:\partial\Omega\setminus C\longrightarrow\partial C
		\]
		and the first hitting time $t$ are well defined.
		
		\item Let $x_0\in\partial\Omega\setminus C$ be
		return-transverse. Then $t$ and $\pi$ are locally Lipschitz in a
		neighbourhood of $x_0$.
		
		\item If the boundaries are smoother near $x_0$ and
		$\pi(x_0)$, the corresponding local regularity of $t$ and $\pi$
		follows from the implicit-function theorem, subject to the same
		transversality condition.
	\end{enumerate}
\end{proposition}

\begin{proof}
	The first assertion follows from Section~4.1.
	
	For the local regularity, set
	\[
	y_0=\pi(x_0),
	\qquad
	t_0=t(x_0).
	\]
	Choose a $C^{1,1}$ defining function $\rho_C$ for $\partial C$ in a
	neighbourhood of $y_0$, normalized so that
	\[
	\nabla\rho_C(y_0)=\nu(y_0).
	\]
	Consider
	\begin{equation}
	H(x,t)
	=
	\rho_C\bigl(x+t\,n(x)\bigr),
	\label{eq:reciprocal-implicit-equation}
	\end{equation}
	for $x$ on $\partial\Omega$ near $x_0$. At $(x_0,t_0)$,
	\[
	H(x_0,t_0)=0,
	\]
	and
	\begin{equation}
	\partial_t H(x_0,t_0)
	=
	\left\langle
	\nu(y_0),n(x_0)
	\right\rangle
	\neq0.
	\label{eq:reciprocal-implicit-derivative}
	\end{equation}
	
	Since $\partial C$ and $\partial\Omega\setminus C$ are $C^{1,1}$,
	the normal field $n$ and the local defining function $\rho_C$ are
	Lipschitz. After restricting to a sufficiently small neighbourhood,
	the derivative in
	\eqref{eq:reciprocal-implicit-derivative} remains bounded away from
	zero. Thus there exists $\eta>0$ such that
	\[
	|\partial_tH(x,t)|\geq\eta
	\]
	in a sufficiently small neighbourhood of $(x_0,t_0)$.
	Since $H$ is uniformly Lipschitz in the $x$-variable there,
	with some Lipschitz constant $L$, the identities
	\[
	H(x_i,t(x_i))=0,
	\qquad i=1,2,
	\]
	give
	\[
	\eta\,|t(x_1)-t(x_2)|
	\leq
	L|x_1-x_2|.
	\]
	Hence $t$ is locally Lipschitz. Since
	\[
	\pi(x)=x+t(x)n(x)
	\]
	and $n$ is locally Lipschitz, $\pi$ is locally Lipschitz as well.
	
	The final assertion follows from the classical implicit-function
	theorem under the corresponding higher regularity assumptions.
\end{proof}
\subsection{The Round-Trip}

The radial and reciprocal maps combine to produce a geometric
round-trip between the two boundaries,
\begin{equation}
\partial C
\xrightarrow{\ \Phi\ }
\partial\Omega\setminus C
\xrightarrow{\ \pi\ }
\partial C.
\label{eq:geometric-roundtrip}
\end{equation}
We define the associated return map by
\begin{equation}
F:=\pi\circ\Phi:
\partial C\longrightarrow\partial C,
\qquad
F(c)=\pi(\Phi(c)).
\label{eq:return-map-definition}
\end{equation}

For $c\in\partial C$, set
\[
x=\Phi(c),
\qquad
c'=F(c)=\pi(x).
\]
The two legs of the construction are then
\begin{equation}
x=c+d(c)\nu(c),
\qquad
c'=x+t(x)n(x).
\label{eq:roundtrip-legs}
\end{equation}

The outward radial structure is illustrated in
Figure~\ref{fig:radial-geometry}, while the inward return mechanism
and the induced return map are shown in
Figure~\ref{fig:roundtrip-maps}.

\begin{figure}[t]
	\centering
	\resizebox{0.92\columnwidth}{!}{%
		\begin{tikzpicture}[
		>=Latex,
		line cap=round,
		line join=round,
		every node/.style={font=\small}
		]
		
		\draw[thick] (0.4,0.1) ellipse [x radius=3.9, y radius=2.2];
		
		\draw[thick] (0,0) ellipse [x radius=1.8, y radius=1.02];
		
		\node at (-1.0,-1.38) {$\partial C$};
		\node at (2.45,-2.08) {$\partial\Omega\setminus C$};
		
		\coordinate (c) at (1.48,0.50);
		\coordinate (x) at (3.35,1.18);
		\coordinate (cp) at (0.18,0.98);
		
		\fill (c) circle (1.5pt);
		\fill (x) circle (1.5pt);
		\fill (cp) circle (1.5pt);
		
		\node[below right] at (c) {$c$};
		\node[above right] at (x) {$x=\Phi(c)$};
		\node[above left] at (cp) {$c'=\pi(x)=F(c)$};
		
		\draw[->,thick] (c) -- (x)
		node[pos=0.52, above=6pt] {$\Phi$};
		
		\draw[->,thick] (x) -- (cp)
		node[pos=0.45, above=6pt] {$\pi$};
		
		\draw[->,dashed,thick] (c) to[bend left=36]
		node[pos=0.56, left=6pt] {$F=\pi\circ\Phi$} (cp);
		
		\draw[->] (x) -- ++(-0.95,-0.28)
		node[below] {$n(x)$};
		
		\node at (0.50,2.00) {$c'-x=t(x)n(x)$};
		
		\end{tikzpicture}%
	}
	\caption{Reciprocal map and return map. Starting from
		$x=\Phi(c)\in\partial\Omega\setminus C$, the reciprocal map follows
		the inward normal $n(x)$ to its first intersection
		$c'=\pi(x)\in\partial C$. The induced return map on the convex core is
		$F=\pi\circ\Phi$.}
	\label{fig:roundtrip-maps}
\end{figure}

Figure~\ref{fig:radial-geometry} emphasizes the role of the metric
projection, whereas Figure~\ref{fig:roundtrip-maps} highlights the
distinction between the reciprocal map and the return map. By
\eqref{eq:radial-map-inverse},
\begin{equation}
P_C(\Phi(c))=c,
\label{eq:projection-after-radial-map}
\end{equation}
whereas in general
\begin{equation}
\pi(\Phi(c))=F(c)\neq c.
\label{eq:nontrivial-return}
\end{equation}
Thus the outward and inward legs are governed by different normal
geometries: the first by the convex core $\partial C$, the second by
the outer boundary $\partial\Omega$.

The discrepancy between these two normal directions is the source
of the nontrivial return dynamics studied below.
\section{The Return Map and Discrete Dynamics}

The geometric construction of Section~4 defines a self-map
\[
F=\pi\circ\Phi:\partial C\longrightarrow\partial C.
\]
We now study the discrete dynamics generated by successive
applications of $F$. Throughout this section, all identities are
exact unless an approximation is explicitly stated.

\subsection{Discrete Iteration and Notation}

Let $c_0\in\partial C$. The return dynamics is defined recursively by
\begin{equation}
c_{k+1}=F(c_k),
\qquad k\geq0.
\label{eq:return-iteration}
\end{equation}

For each iterate $c_k$, we introduce the associated point on the
outer boundary
\begin{equation}
x_k:=\Phi(c_k)
=
c_k+d(c_k)\nu(c_k).
\label{eq:outer-iterate}
\end{equation}
For brevity, we write
\begin{equation}
d_k:=d(c_k),
\;
t_k:=t(x_k),
\;
\nu_k:=\nu(c_k),
\;
n_k:=n(x_k).
\label{eq:iterate-notation}
\end{equation}

Since
\[
c_{k+1}
=
\pi(x_k),
\]
the $k$-th step of the dynamics consists of the geometric excursion
\begin{equation}
c_k
\xrightarrow{\ \Phi\ }
x_k
\xrightarrow{\ \pi\ }
c_{k+1}.
\label{eq:kth-roundtrip}
\end{equation}

The next subsection derives the exact displacement associated with
one such step.
\subsection{Fundamental Geometric Identity}

The two legs of the $k$-th excursion give an exact relation between
successive points of the return dynamics.

\begin{proposition}[Exact displacement identity]
	\label{prop:exact-displacement}
	For every $k\geq0$,
	\begin{equation}
	c_{k+1}-c_k
	=
	d_k\nu_k+t_k n_k .
	\label{eq:exact-displacement}
	\end{equation}
\end{proposition}

\begin{proof}
	By definition of the radial map,
	\[
	x_k=c_k+d_k\nu_k.
	\]
	By definition of the reciprocal map,
	\[
	c_{k+1}=x_k+t_k n_k.
	\]
	Substituting the first identity into the second gives
	\[
	c_{k+1}
	=
	c_k+d_k\nu_k+t_k n_k,
	\]
	which is equivalent to
	\eqref{eq:exact-displacement}.
\end{proof}

Equation~\eqref{eq:exact-displacement} is the basic closure relation
for one step of the return dynamics. It contains no approximation
and separates the outward radial displacement $d_k\nu_k$ from the
inward normal displacement $t_k n_k$.

Taking scalar products with $\nu_k$ and $n_k$ gives the exact
relations
\begin{equation}
\left\langle c_{k+1}-c_k,\nu_k\right\rangle
=
d_k+t_k\left\langle n_k,\nu_k\right\rangle,
\label{eq:exact-normal-projection}
\end{equation}
and
\begin{equation}
\left\langle c_{k+1}-c_k,n_k\right\rangle
=
d_k\left\langle\nu_k,n_k\right\rangle+t_k.
\label{eq:exact-inward-projection}
\end{equation}

These identities will be used below when estimating the return
distance and the displacement. In particular, neither
\eqref{eq:exact-normal-projection} nor
\eqref{eq:exact-inward-projection} determines $t_k$ from $d_k$ and
the angle between the normals alone, because the displacement
$c_{k+1}-c_k$ is itself nonzero in general.
\subsection{Return Distance and Exact Closure Condition}

The return distance $t_k$ is determined by the condition that the
inward normal ray from $x_k$ meets the convex core for the first
time at $c_{k+1}$. Thus
\begin{equation}
c_{k+1}
=
c_k+d_k\nu_k+t_k n_k
\in\partial C,
\label{eq:return-closure}
\end{equation}
with
\begin{equation}
t_k
=
\min\left\{
t>0:
c_k+d_k\nu_k+t n_k\in C
\right\}.
\label{eq:return-distance-exact}
\end{equation}

Equivalently, if $\rho_C$ is a local defining function for
$\partial C$ near $c_{k+1}$, chosen so that
\[
\rho_C=0 \quad\text{on }\partial C,
\]
then $t_k$ satisfies the exact implicit equation
\begin{equation}
\rho_C\!\left(
c_k+d_k\nu_k+t_k n_k
\right)
=
0.
\label{eq:return-distance-implicit}
\end{equation}

The following identity records the relation between the return
distance and the displacement along $\partial C$.

\begin{proposition}[Exact return-distance identity]
	\label{prop:return-distance-identity}
	For every $k\geq0$,
	\begin{equation}
	t_k
	=
	\frac{
		d_k-
		\left\langle c_{k+1}-c_k,\nu_k\right\rangle
	}{
		-\left\langle n_k,\nu_k\right\rangle
	}.
	\label{eq:exact-return-distance}
	\end{equation}
	Here
	\[
	-\langle n_k,\nu_k\rangle>0
	\]
	by the orientation and radial transversality of the intersection
	$x_k=\Phi(c_k)$.
\end{proposition}

\begin{proof}
	Taking the scalar product of the exact displacement identity
	\eqref{eq:exact-displacement} with $\nu_k$ gives
	\[
	\left\langle c_{k+1}-c_k,\nu_k\right\rangle
	=
	d_k+t_k\left\langle n_k,\nu_k\right\rangle.
	\]
	Solving for $t_k$ yields
	\eqref{eq:exact-return-distance}.
\end{proof}

Since $\nu_k$ is the outward normal to the supporting hyperplane of
the convex set $C$ at $c_k$, and since $c_{k+1}\in C$, one has
\begin{equation}
\left\langle
c_{k+1}-c_k,\nu_k
\right\rangle
\leq0.
\label{eq:convex-support-inequality}
\end{equation}
Consequently,
\begin{equation}
t_k
\geq
\frac{d_k}
{-\langle n_k,\nu_k\rangle}
=
\frac{d_k}
{|\langle n_k,\nu_k\rangle|}.
\label{eq:return-distance-lower-bound}
\end{equation}
In particular, the return distance is not determined by $d_k$ and
the angle between $n_k$ and $\nu_k$ alone. The additional term
\[
\left\langle c_{k+1}-c_k,\nu_k\right\rangle
\]
measures the normal component, relative to the supporting
hyperplane at $c_k$, of the displacement of the return point on
$\partial C$.

At a fixed point, where $c_{k+1}=c_k$, the correction term vanishes
and \eqref{eq:exact-return-distance} reduces to
\[
t_k
=
\frac{d_k}
{|\langle n_k,\nu_k\rangle|}.
\]
As will be shown later, at such a point $n_k=-\nu_k$, and therefore
$t_k=d_k$.
\section{First-Order Expansion of the Return Map}

We now derive the local expansion of the return map from the exact
geometry of the radial parametrization. The derivation is organized
so that exact identities are separated from asymptotic statements.
Throughout this section, all differential identities are understood
at points where the relevant derivatives exist. Under the
$C^{1,1}$ and Lipschitz assumptions introduced above, this holds
almost everywhere on $\partial C$.

\subsection{Exact Normal to the Outer Boundary}

Let $c\in\partial C$ be a point at which $d$ and $\nu$ are
differentiable, and set
\[
x=\Phi(c),
\qquad
g_c:=\nabla_{\partial C}d(c).
\]
Recall from \eqref{eq:Lc-definition} that
\[
L_c=I-d(c)S_c.
\]
Since the eigenvalues of $L_c$ are
\[
1+d(c)\kappa_i(c)>0,
\]
the operator $L_c:T_c\partial C\to T_c\partial C$ is invertible.

Define the tangent vector
\begin{equation}
w_c
:=
L_c^{-1}g_c
=
\bigl(I-d(c)S_c\bigr)^{-1}
\nabla_{\partial C}d(c).
\label{eq:wc-definition}
\end{equation}

\begin{proposition}[Exact normal formula]
	\label{prop:exact-normal}
	At $x=\Phi(c)$, the outward unit normal to
	$\partial\Omega$ is
	\begin{equation}
	N_\Omega(x)
	=
	\frac{\nu(c)-w_c}
	{\sqrt{1+|w_c|^2}},
	\label{eq:exact-outward-normal}
	\end{equation}
	and therefore the inward unit normal is
	\begin{equation}
	n(x)
	=
	\frac{-\nu(c)+w_c}
	{\sqrt{1+|w_c|^2}}.
	\label{eq:exact-inward-normal}
	\end{equation}
\end{proposition}

\begin{proof}
	By \eqref{eq:differential-radial-map}, for every
	$v\in T_c\partial C$,
	\[
	D\Phi(c)[v]
	=
	L_cv+\langle g_c,v\rangle\nu(c).
	\]
	Since the shape operator $S_c$ is self-adjoint on $T_c\partial C$,
	the operator
	\[
	L_c=I-d(c)S_c
	\]
	is self-adjoint as well. Hence, using
	$L_cw_c=g_c$,
	\[
	\begin{aligned}
	\left\langle
	\nu(c)-w_c,
	D\Phi(c)[v]
	\right\rangle
	&=
	-\langle w_c,L_cv\rangle
	+\langle g_c,v\rangle\\
	&=
	-\langle L_cw_c,v\rangle
	+\langle g_c,v\rangle\\
	&=0.
	\end{aligned}
	\]
	Hence $\nu(c)-w_c$ is normal to the image hypersurface
	$\partial\Omega$ at $x=\Phi(c)$.
	
	Moreover, since $w_c\in T_c\partial C$,
	\[
	|\nu(c)-w_c|^2
	=
	1+|w_c|^2.
	\]
	Thus normalization gives \eqref{eq:exact-outward-normal}.
	
	Finally,
	\[
	\left\langle
	N_\Omega(x),\nu(c)
	\right\rangle
	=
	\frac{1}{\sqrt{1+|w_c|^2}}>0,
	\]
	so this choice has the outward orientation. Its negative is the
	inward unit normal, which gives
	\eqref{eq:exact-inward-normal}.
\end{proof}
\begin{remark}
	\label{rem:normal-parallel-case}
	If $\nabla_{\partial C}d(c)=0$, then $w_c=0$ and hence
	\[
	N_\Omega(\Phi(c))=\nu(c),
	\qquad
	n(\Phi(c))=-\nu(c).
	\]
\end{remark}
\subsection{Exact Decomposition and Closure Equation}

The exact normal formula reduces the return geometry to a scalar
closure problem. Let $c\in\partial C$ be a point at which the
quantities introduced in Section~6.1 are defined, and set
\[
x=\Phi(c),
\qquad
t_c:=t(\Phi(c)).
\]
For brevity, write
\[
d=d(c),
\qquad
\nu=\nu(c),
\qquad
w=w_c.
\]

We introduce the scaled return parameter
\begin{equation}
\sigma(c)
:=
\frac{t_c}{\sqrt{1+|w_c|^2}}.
\label{eq:sigma-definition}
\end{equation}

\begin{proposition}[Exact decomposition of the return displacement]
	\label{prop:exact-return-decomposition}
	The return displacement satisfies
	\begin{equation}
	F(c)-c
	=
	\sigma(c)w_c
	+
	\bigl(d(c)-\sigma(c)\bigr)\nu(c).
	\label{eq:exact-return-decomposition}
	\end{equation}
	In particular,
	\begin{align}
	P_{T_c\partial C}\bigl(F(c)-c\bigr)
	&=
	\sigma(c)w_c,
	\label{eq:exact-tangential-displacement}
	\\
	\left\langle
	F(c)-c,\nu(c)
	\right\rangle
	&=
	d(c)-\sigma(c).
	\label{eq:exact-normal-displacement}
	\end{align}
\end{proposition}

\begin{proof}
	By the exact displacement identity
	\eqref{eq:exact-displacement},
	\[
	F(c)-c
	=
	d(c)\nu(c)
	+
	t_c\,n(\Phi(c)).
	\]
	Using \eqref{eq:exact-inward-normal},
	\[
	t_c\,n(\Phi(c))
	=
	\frac{t_c}{\sqrt{1+|w_c|^2}}
	\bigl(-\nu(c)+w_c\bigr).
	\]
	The definition \eqref{eq:sigma-definition} therefore gives
	\[
	F(c)-c
	=
	d(c)\nu(c)
	+
	\sigma(c)\bigl(-\nu(c)+w_c\bigr),
	\]
	which is precisely
	\eqref{eq:exact-return-decomposition}. Since
	$w_c\in T_c\partial C$ and
	$\nu(c)\perp T_c\partial C$, the two component identities follow.
\end{proof}
\begin{lemma}[Local return near a critical point]
	\label{lem:local-return-critical}
	Let $c_*\in\partial C$ be a point at which $d$ is differentiable and
	\[
	\nabla_{\partial C}d(c_*)=0.
	\]
	Then
	\begin{equation}
	F(c_*)=c_*.
	\label{eq:critical-point-fixed-preliminary}
	\end{equation}
	Moreover, the inward normal ray from
	$x_*=\Phi(c_*)$ meets $\partial C$ transversely at $c_*$.
	Consequently, there exist neighbourhoods
	$V\subset\partial C$ of $c_*$ and
	$U\subset\partial C$ of $c_*$ such that
	\[
	F(V)\subset U,
	\]
	where $U$ is contained in a single local coordinate chart of
	$\partial C$.
\end{lemma}

\begin{proof}
	Since
	\[
	\nabla_{\partial C}d(c_*)=0,
	\]
	equation~\eqref{eq:wc-definition} gives $w_{c_*}=0$. Hence the exact
	normal formula \eqref{eq:exact-inward-normal} yields
	\[
	n(\Phi(c_*))=-\nu(c_*).
	\]
	Therefore the inward normal ray issued from
	\[
	\Phi(c_*)=c_*+d(c_*)\nu(c_*)
	\]
	is the radial segment directed toward $c_*$. Its first intersection
	with the convex core is $c_*$, and consequently
	\[
	F(c_*)=c_*.
	\]
	
	Furthermore,
	\[
	\left\langle
	n(\Phi(c_*)),\nu(c_*)
	\right\rangle=-1,
	\]
	so the intersection is transverse. Proposition~\ref{prop:reciprocal-basic-properties}
	therefore gives local continuity of the reciprocal map near
	$\Phi(c_*)$. Since the radial map $\Phi$ is locally continuous,
	the return map $F=\pi\circ\Phi$ is continuous near $c_*$. The final
	assertion follows by choosing a sufficiently small coordinate
	neighbourhood $U$ of $c_*$ and then shrinking $V$ so that
	$F(V)\subset U$.
\end{proof}
We now restrict to the local-return regime provided by
Lemma~\ref{lem:local-return-critical}. Let $c_*$ be the critical
point appearing there. After shrinking the neighbourhood $V$ of
$c_*$ if necessary, the $C^{1,1}$ regularity of $\partial C$
provides a radius $r>0$ such that, for every $c\in V$, the portion
of $\partial C$ near $c$ can be represented as a normal graph
\begin{equation}
u\longmapsto
c+u+h_c(u)\nu(c),
\qquad
u\in U_c\subset T_c\partial C,
\label{eq:local-core-graph}
\end{equation}
where $U_c$ contains the tangent ball of radius $r$ and
\begin{equation}
h_c(0)=0,
\qquad
Dh_c(0)=0.
\label{eq:local-core-graph-origin}
\end{equation}
Moreover, after possibly shrinking $V$ once more, the $C^{1,1}$
bounds may be chosen uniformly:
\begin{equation}
|h_c(u)|
\leq
K|u|^2,
\qquad
c\in V,\quad u\in U_c,
\label{eq:local-core-quadratic-bound}
\end{equation}
for some constant $K>0$.

Since $F(c_*)=c_*$ and $F$ is continuous near $c_*$ by
Lemma~\ref{lem:local-return-critical}, one has
\[
|F(c)-c|\longrightarrow0
\qquad\text{as }c\longrightarrow c_*.
\]
Hence, after shrinking $V$ again if necessary, $F(c)$ belongs to
the above normal-graph neighbourhood based at $c$ for every
$c\in V$.

At every such $c$ at which $d$ and $\nu$ are differentiable, there
is therefore a unique $u_c\in U_c$ such that
\begin{equation}
F(c)-c
=
u_c+h_c(u_c)\nu(c).
\label{eq:return-local-graph}
\end{equation}

On the other hand, the exact return decomposition
\eqref{eq:exact-return-decomposition} gives
\[
F(c)-c
=
\sigma(c)w_c+
\bigl(d(c)-\sigma(c)\bigr)\nu(c).
\]
Taking the orthogonal projection onto $T_c\partial C$ and the scalar
product with $\nu(c)$ yields
\[
u_c=\sigma(c)w_c
\]
and
\[
h_c(u_c)=d(c)-\sigma(c).
\]
Consequently,
\begin{equation}
d(c)-\sigma(c)
=
h_c\bigl(\sigma(c)w_c\bigr).
\label{eq:exact-closure-equation}
\end{equation}
Equation~\eqref{eq:exact-closure-equation} immediately implies
\begin{equation}
|d(c)-\sigma(c)|
\leq
K\,\sigma(c)^2|w_c|^2.
\label{eq:sigma-quadratic-estimate}
\end{equation}
Thus the normal component of the return displacement is of second
order in the slope $w_c$.

For a convex core, the local graph lies on the inward side of its
supporting tangent hyperplane, and hence
\[
h_c(u)\leq0
\]
for $u$ sufficiently small. Consequently,
\begin{equation}
\sigma(c)\geq d(c).
\label{eq:sigma-lower-bound}
\end{equation}
Using \eqref{eq:sigma-definition}, this gives
\begin{equation}
t_c
\geq
d(c)\sqrt{1+|w_c|^2}.
\label{eq:return-distance-w-lower-bound}
\end{equation}
\begin{remark}
	\label{rem:closure-parallel-case}
	If $w_c=0$, then \eqref{eq:exact-closure-equation} gives
	$\sigma(c)=d(c)$. Hence
	\[
	t_c=d(c),
	\qquad
	F(c)=c.
	\]
	Thus vanishing tangential slope produces an exact fixed point, not
	merely an approximate one.
\end{remark}
\subsection{First-Order Expansion}

We can now derive the first-order expansion of the return map.
The expansion is naturally expressed in terms of the tangential
slope
\[
w_c
=
\bigl(I-d(c)S_c\bigr)^{-1}
\nabla_{\partial C}d(c).
\]

\begin{theorem}[Quadratic remainder estimate for the return map]
	\label{thm:first-order-return-expansion}
	Let $c_*\in\partial C$ be a critical point as in
	Lemma~\ref{lem:local-return-critical}, and let $V$ be a sufficiently
	small local-return neighbourhood of $c_*$. Then there exists a
	constant $C>0$ such that, at every $c\in V$ where $d$ and $\nu$
	are differentiable,
	\begin{equation}
	F(c)-c
	=
	d(c)w_c+R(c),
	\label{eq:return-expansion-w}
	\end{equation}
	where
	\begin{equation}
	|R(c)|
	\leq
	C|w_c|^2.
	\label{eq:return-remainder-w}
	\end{equation}
	Equivalently,
	\begin{equation}
	F(c)-c
	=
	d(c)
	\bigl(I-d(c)S_c\bigr)^{-1}
	\nabla_{\partial C}d(c)
	+
	R(c),
	\label{eq:first-order-return-expansion}
	\end{equation}
	with
	\begin{equation}
	|R(c)|
	\leq
	C
	\left|
	\nabla_{\partial C}d(c)
	\right|^2.
	\label{eq:first-order-gradient-remainder}
	\end{equation}
	
	Thus, at every differentiability point in the above neighbourhood,
	the return displacement satisfies the quadratic remainder estimate
	\begin{equation}
	F(c)-c
	=
	d(c)
	\bigl(I-d(c)S_c\bigr)^{-1}
	\nabla_{\partial C}d(c)
	+
	O\!\left(
	|\nabla_{\partial C}d(c)|^2
	\right).
	\label{eq:first-order-return-asymptotic}
	\end{equation}
\end{theorem}

\begin{proof}
	By Proposition~\ref{prop:exact-return-decomposition},
	\begin{equation}
	F(c)-c
	=
	\sigma(c)w_c
	+
	\bigl(d(c)-\sigma(c)\bigr)\nu(c).
	\label{eq:first-order-proof-start}
	\end{equation}
	The exact closure equation
	\eqref{eq:exact-closure-equation}, together with the quadratic
	estimate for the local graph of $\partial C$, gives
	\begin{equation}
	|d(c)-\sigma(c)|
	\leq
	K\,\sigma(c)^2|w_c|^2.
	\label{eq:sigma-minus-d-estimate}
	\end{equation}
	
	By Lemma~\ref{lem:local-return-critical} and
	Proposition~\ref{prop:reciprocal-basic-properties}, the return
	distance $t(\Phi(c))$ is locally bounded in $V$. Since
	\[
	\sigma(c)
	=
	\frac{t(\Phi(c))}
	{\sqrt{1+|w_c|^2}},
	\]
	there exists $M>0$ such that
	\[
	0<\sigma(c)\leq M
	\]
	throughout $V$. Hence
	\begin{equation}
	|d(c)-\sigma(c)|
	\leq
	KM^2|w_c|^2.
	\label{eq:sigma-minus-d-uniform}
	\end{equation}
	
	Subtracting $d(c)w_c$ from
	\eqref{eq:first-order-proof-start} yields
	\[
	\begin{aligned}
	F(c)-c-d(c)w_c
	&=
	\bigl(\sigma(c)-d(c)\bigr)w_c
	+
	\bigl(d(c)-\sigma(c)\bigr)\nu(c)
	\\
	&=
	\bigl(\sigma(c)-d(c)\bigr)
	\bigl(w_c-\nu(c)\bigr).
	\end{aligned}
	\]
	Therefore
	\[
	\begin{aligned}
	\left|
	F(c)-c-d(c)w_c
	\right|
	&\leq
	|d(c)-\sigma(c)|
	\sqrt{1+|w_c|^2}.
	\end{aligned}
	\]
	After shrinking $V$ if necessary, $w_c$ is bounded at the
	differentiability points under consideration. Combining this with
	\eqref{eq:sigma-minus-d-uniform} gives
	\[
	\left|
	F(c)-c-d(c)w_c
	\right|
	\leq
	C|w_c|^2,
	\]
	which proves \eqref{eq:return-expansion-w} and
	\eqref{eq:return-remainder-w}.
	
	It remains to express the estimate in terms of the thickness
	gradient. Since the eigenvalues of
	\[
	L_c=I-d(c)S_c
	\]
	are
	\[
	1+d(c)\kappa_i(c)\geq1,
	\]
	one has
	\begin{equation}
	\|L_c^{-1}\|\leq1.
	\label{eq:L-inverse-bound}
	\end{equation}
	Consequently,
	\[
	|w_c|
	=
	\left|
	L_c^{-1}\nabla_{\partial C}d(c)
	\right|
	\leq
	\left|
	\nabla_{\partial C}d(c)
	\right|.
	\]
	Substitution into the preceding estimate gives
	\eqref{eq:first-order-gradient-remainder} and completes the proof.
\end{proof}
It is convenient to introduce the positive tangential operator
\begin{equation}
B_c
:=
d(c)\bigl(I-d(c)S_c\bigr)^{-1}.
\label{eq:preconditioner-B}
\end{equation}
Its eigenvalues in the principal directions of $\partial C$ are
\begin{equation}
\beta_i(c)
=
\frac{d(c)}
{1+d(c)\kappa_i(c)}
>0.
\label{eq:B-eigenvalues}
\end{equation}
Hence the quadratic remainder estimate takes the compact form
\begin{equation}
F(c)-c
=
B_c\nabla_{\partial C}d(c)
+
O\!\left(
|\nabla_{\partial C}d(c)|^2
\right).
\label{eq:return-expansion-B}
\end{equation}

\begin{remark}
	\label{rem:first-order-regularity}
	Under the $C^{1,1}$ assumptions used above,
	Theorem~\ref{thm:first-order-return-expansion} is a
	quadratic remainder estimate at points where $d$ and $\nu$
	are differentiable, hence almost everywhere on $\partial C$.
	
	If, in addition, $d$ is of class $C^1$ in a neighbourhood
	of the critical point $c_*$, then
	\[
	\nabla_{\partial C}d(c)\longrightarrow0
	\qquad
	\text{as }c\longrightarrow c_*.
	\]
	In this case the estimate becomes a genuine local asymptotic
	expansion:
	\[
	F(c)-c
	=
	B_c\nabla_{\partial C}d(c)
	+
	o\!\left(
	|\nabla_{\partial C}d(c)|
	\right),
	\]
	along points where the shape operator is defined.
	
	If $\partial C$ is moreover of class $C^2$ near $c_*$,
	then $S_c$ is defined continuously and the expansion holds
	in the classical pointwise sense.
\end{remark}
\subsection{Circle and Sphere Consistency Checks}

The sign and curvature dependence in
Theorem~\ref{thm:first-order-return-expansion} can be checked
explicitly when the convex core is a Euclidean sphere.

\subsubsection*{Spherical core}

Let
\[
C=\overline{B_R(0)}\subset\R^N,
\qquad
\partial C=R\mathbb S^{N-1}.
\]
For the outward orientation,
\[
\nu(c)=\frac{c}{R},
\qquad
S_c=-\frac{1}{R}I
\]
on $T_c\partial C$. Hence
\begin{equation}
L_c
=
I-d(c)S_c
=
\left(1+\frac{d(c)}{R}\right)I,
\label{eq:sphere-L}
\end{equation}
and therefore
\begin{equation}
w_c
=
\frac{R}{R+d(c)}
\nabla_{\partial C}d(c).
\label{eq:sphere-w}
\end{equation}

The first-order expansion
\eqref{eq:first-order-return-expansion} becomes
\begin{equation}
F(c)-c
=
\frac{d(c)R}{R+d(c)}
\nabla_{\partial C}d(c)
+
O\!\left(
|\nabla_{\partial C}d(c)|^2
\right).
\label{eq:sphere-return-expansion}
\end{equation}
In particular, the leading displacement has the same direction as
the tangential gradient of the thickness function.

The closure equation can also be checked directly. In normal
coordinates based at $c$, the sphere is represented by
\begin{equation}
h_c(u)
=
-R+\sqrt{R^2-|u|^2}.
\label{eq:sphere-local-graph}
\end{equation}
Thus
\begin{equation}
h_c(u)
=
-\frac{|u|^2}{2R}
+
O(|u|^4).
\label{eq:sphere-local-graph-expansion}
\end{equation}
Substitution into
\eqref{eq:exact-closure-equation} gives
\begin{equation}
d(c)-\sigma(c)
=
-R+
\sqrt{
	R^2-\sigma(c)^2|w_c|^2
}.
\label{eq:sphere-closure}
\end{equation}
Equivalently,
\begin{equation}
\bigl(R+d(c)-\sigma(c)\bigr)^2
+
\sigma(c)^2|w_c|^2
=
R^2.
\label{eq:sphere-closure-quadratic}
\end{equation}
For $|w_c|\to0$, the branch corresponding to the first intersection
with the sphere satisfies
\begin{equation}
\sigma(c)
=
d(c)+O(|w_c|^2),
\label{eq:sphere-sigma-expansion}
\end{equation}
which is consistent with
Theorem~\ref{thm:first-order-return-expansion}.
\subsubsection*{Unit circle}

Consider now the planar case
\[
C=\overline{B_1(0)}\subset\R^2.
\]
Write
\[
c(\theta)
=
(\cos\theta,\sin\theta),
\]
and let $e_\theta$ denote the positively oriented unit tangent.
Then
\[
\nu(\theta)=c(\theta),
\qquad
S=-I,
\]
and
\[
\nabla_{\partial C}d(\theta)
=
d'(\theta)e_\theta.
\]
Consequently,
\begin{equation}
w(\theta)
=
\frac{d'(\theta)}
{1+d(\theta)}
e_\theta.
\label{eq:circle-w}
\end{equation}

Introduce
\begin{equation}
p(\theta)
:=
\frac{d'(\theta)}
{1+d(\theta)}.
\label{eq:circle-p}
\end{equation}
The outer point is
\[
x(\theta)
=
\bigl(1+d(\theta)\bigr)\nu(\theta),
\]
and the exact inward normal formula gives
\begin{equation}
n(x(\theta))
=
\frac{
	-\nu(\theta)+p(\theta)e_\theta
}{
	\sqrt{1+p(\theta)^2}
}.
\label{eq:circle-exact-normal}
\end{equation}

Writing
\[
\sigma
=
\frac{t}{\sqrt{1+p^2}},
\]
the return point is therefore
\begin{equation}
F(c(\theta))
=
(1+d-\sigma)\nu(\theta)
+
\sigma p\,e_\theta.
\label{eq:circle-return-point}
\end{equation}
Since $F(c(\theta))$ belongs to the unit circle, the exact closure
condition is
\begin{equation}
(1+d-\sigma)^2
+
\sigma^2p^2
=
1.
\label{eq:circle-exact-closure}
\end{equation}
Solving the corresponding quadratic equation and choosing the
branch giving the first intersection yields
\begin{equation}
\sigma
=
\frac{
	1+d-
	\sqrt{1-d(2+d)p^2}
}{
	1+p^2}.
\label{eq:circle-exact-sigma}
\end{equation}
Here $d=d(\theta)$ and $p=p(\theta)$.
For $p\to0$,
\begin{equation}
\sigma
=
d+O(p^2).
\label{eq:circle-sigma-expansion}
\end{equation}
Let
\[
F(c(\theta))
=
c(\theta+\delta(\theta)).
\]
Taking the tangential component of
\eqref{eq:circle-return-point} gives
\begin{equation}
\sin\delta(\theta)
=
\sigma(\theta)p(\theta).
\label{eq:circle-angle-exact}
\end{equation}
Using
\eqref{eq:circle-sigma-expansion} and
\eqref{eq:circle-p}, we obtain
\begin{equation}
\delta(\theta)
=
\frac{
	d(\theta)
}{
	1+d(\theta)
}
d'(\theta)
+
O\!\left(
|d'(\theta)|^2
\right).
\label{eq:circle-angle-expansion}
\end{equation}
Equivalently, in ambient coordinates,
\begin{equation}
F(c(\theta))-c(\theta)
=
\frac{
	d(\theta)
}{
	1+d(\theta)
}
d'(\theta)e_\theta
+
O\!\left(
|d'(\theta)|^2
\right).
\label{eq:circle-return-expansion}
\end{equation}
Thus the unit-circle calculation reproduces exactly the leading
term predicted by
Theorem~\ref{thm:first-order-return-expansion}. In particular, its
coefficient is positive and equals
\[
\frac{d(\theta)}{1+d(\theta)}.
\]
The circle therefore provides a direct
consistency check both for the sign and for the curvature factor in
the return-map expansion.

\subsection{Geometric Interpretation}

The expansion
\eqref{eq:return-expansion-B} shows that, to first order, the return
displacement is obtained by applying the positive tangential
operator
\[
B_c
=
d(c)\bigl(I-d(c)S_c\bigr)^{-1}
\]
to the tangential gradient of the thickness function:
\begin{equation}
F(c)-c
=
B_c\nabla_{\partial C}d(c)
+
O\!\left(
|\nabla_{\partial C}d(c)|^2
\right).
\label{eq:return-geometric-interpretation}
\end{equation}

Thus the leading-order dynamics is not an ordinary gradient step,
but a curvature-dependent preconditioned gradient step on
$\partial C$.

In a principal direction $e_i(c)$ of the convex core,
\[
S_ce_i(c)=-\kappa_i(c)e_i(c),
\]
and hence
\begin{equation}
B_ce_i(c)
=
\frac{d(c)}
{1+d(c)\kappa_i(c)}
e_i(c).
\label{eq:principal-preconditioning}
\end{equation}
Therefore curvature modifies the strength of the return displacement
anisotropically. Directions with larger principal curvature are
weighted by the smaller factor
\[
\frac{d(c)}
{1+d(c)\kappa_i(c)}.
\]

Since every eigenvalue of $B_c$ is positive,
\begin{equation}
\left\langle
B_c\nabla_{\partial C}d(c),
\nabla_{\partial C}d(c)
\right\rangle
\geq0,
\label{eq:positive-gradient-pairing}
\end{equation}
with strict inequality whenever
$\nabla_{\partial C}d(c)\neq0$. Hence the leading-order displacement
points in a gradient-ascent direction for the thickness function,
with the metric distorted by the curvature of the convex core.

This interpretation should be understood locally. The expansion
does not by itself imply that
\[
d(F(c))>d(c)
\]
for every noncritical point, since such a monotonicity statement
also depends on the size of the step and on higher-order terms.
Accordingly, no global Lyapunov conclusion is drawn at this stage.

At a critical point $c_*$ of $d$,
\[
\nabla_{\partial C}d(c_*)=0,
\]
the leading-order displacement vanishes. As shown in
Lemma~\ref{lem:local-return-critical}, this is in fact an exact
statement:
\begin{equation}
\nabla_{\partial C}d(c_*)=0
\quad\Longrightarrow\quad
F(c_*)=c_*.
\label{eq:critical-implies-fixed}
\end{equation}
The converse and the local stability of such fixed points will be
studied in the next section.
\section{Fixed Points, Linearization, and Local Stability}

We now derive the local dynamical consequences of the
return-map expansion. The exact geometry established in Section~6
first identifies the fixed points of the return map. For the
subsequent linearization and stability analysis, stronger local
smoothness assumptions will be introduced explicitly when needed.

\subsection{Fixed Points and Critical Points}

The exact decomposition of the return displacement gives a direct
characterization of the fixed points.

\begin{theorem}[Fixed points and critical points]
	\label{thm:fixed-critical-equivalence}
	Let $c\in\partial C$ be a point at which $d$ and $\nu$ are
	differentiable. Then
	\begin{equation}
	F(c)=c
	\quad\Longleftrightarrow\quad
	\nabla_{\partial C}d(c)=0.
	\label{eq:fixed-critical-equivalence}
	\end{equation}
\end{theorem}

\begin{proof}
	Suppose first that
	\[
	\nabla_{\partial C}d(c)=0.
	\]
	By \eqref{eq:wc-definition},
	\[
	w_c=0.
	\]
	The exact inward-normal formula
	\eqref{eq:exact-inward-normal} therefore gives
	\[
	n(\Phi(c))=-\nu(c).
	\]
	Hence the inward normal ray from
	\[
	\Phi(c)=c+d(c)\nu(c)
	\]
	runs back along the same normal line toward $c$. Since $c$ is the
	metric projection of $\Phi(c)$ onto the convex set $C$, its first
	intersection with $C$ is $c$. Consequently,
	\[
	F(c)=c.
	\]
	
	Conversely, suppose that
	\[
	F(c)=c.
	\]
	By the exact return decomposition
	\eqref{eq:exact-return-decomposition},
	\[
	0
	=
	\sigma(c)w_c
	+
	\bigl(d(c)-\sigma(c)\bigr)\nu(c).
	\]
	The first term is tangent to $\partial C$ at $c$, whereas the second
	is normal to $\partial C$. Since these subspaces are orthogonal,
	both components must vanish:
	\begin{equation}
	\sigma(c)w_c=0,
	\qquad
	d(c)-\sigma(c)=0.
	\label{eq:fixed-component-vanishing}
	\end{equation}
	Because $t(\Phi(c))>0$, the definition
	\[
	\sigma(c)
	=
	\frac{t(\Phi(c))}
	{\sqrt{1+|w_c|^2}}
	\]
	implies $\sigma(c)>0$. Hence
	\[
	w_c=0.
	\]
	Since
	\[
	w_c
	=
	\bigl(I-d(c)S_c\bigr)^{-1}
	\nabla_{\partial C}d(c)
	\]
	and $I-d(c)S_c$ is invertible, we conclude that
	\[
	\nabla_{\partial C}d(c)=0.
	\]
	This proves \eqref{eq:fixed-critical-equivalence}.
\end{proof}
\begin{remark}
	\label{rem:fixed-point-regularity}
	Under the $C^{1,1}$ framework of the preceding sections,
	Theorem~\ref{thm:fixed-critical-equivalence} is an almost-everywhere
	statement, because $d$ and the normal field are differentiable
	almost everywhere.
	
	If $d$ is $C^1$ and $\partial C$ is $C^2$ in a neighbourhood of a
	point under consideration, the equivalence
	\[
	F(c)=c
	\quad\Longleftrightarrow\quad
	\nabla_{\partial C}d(c)=0
	\]
	holds in the classical pointwise sense.
\end{remark}
\subsection{Linearization at a Critical Point}

We now pass from the $C^{1,1}$ framework used for the first-order
expansion to the classical differentiable setting required for a
pointwise linearization. Throughout this subsection, assume that
$\partial C$ is of class $C^2$ and that the thickness function $d$
is of class $C^2$ in a neighbourhood of a critical point
$c_*\in\partial C$.

Set
\[
d_*:=d(c_*),
\qquad
S_*:=S_{c_*},
\]
and define
\begin{equation}
B_*
:=
d_*
\bigl(I-d_*S_*\bigr)^{-1}.
\label{eq:B-star}
\end{equation}

Since $c_*$ is critical,
\[
\nabla_{\partial C}d(c_*)=0,
\]
and Theorem~\ref{thm:fixed-critical-equivalence} gives
\[
F(c_*)=c_*.
\]

To avoid ambiguity between the bilinear Hessian and its associated
endomorphism, let
\[
H_*:T_{c_*}\partial C\longrightarrow T_{c_*}\partial C
\]
denote the self-adjoint operator determined by
\begin{equation}
\left\langle H_*v,z\right\rangle
=
\operatorname{Hess}_{\partial C}d(c_*)[v,z],
\qquad
v,z\in T_{c_*}\partial C.
\label{eq:hessian-operator}
\end{equation}

\begin{theorem}[Linearization at a critical point]
	\label{thm:critical-linearization}
	Under the assumptions above, the derivative of the return map at
	$c_*$ is
	\begin{equation}
	DF(c_*)\big|_{T_{c_*}\partial C}
	=
	I+B_*H_*.
	\label{eq:critical-linearization}
	\end{equation}
	Equivalently,
	\begin{equation}
	DF(c_*)\big|_{T_{c_*}\partial C}
	=
	I+
	d_*
	\bigl(I-d_*S_*\bigr)^{-1}
	H_*.
	\label{eq:critical-linearization-expanded}
	\end{equation}
\end{theorem}

\begin{proof}
	Choose a local $C^2$ coordinate parametrization
	\[
	c:U\subset T_{c_*}\partial C\longrightarrow\partial C
	\]
	such that
	\[
	c(0)=c_*,
	\qquad
	Dc(0)=I.
	\]
	After shrinking $U$ if necessary, we identify the nearby tangent
	spaces with $T_{c_*}\partial C$ through the corresponding local
	coordinate frame.
	
	Since $c$ is $C^2$,
	\begin{equation}
	c(u)
	=
	c_*+u+O(|u|^2)
	\qquad
	\text{as }u\to0.
	\label{eq:local-chart-expansion}
	\end{equation}
	
	By Theorem~\ref{thm:first-order-return-expansion},
	\begin{equation}
	F(c(u))-c(u)
	=
	B_{c(u)}
	\nabla_{\partial C}d(c(u))
	+
	R(c(u)),
	\label{eq:linearization-local-expansion}
	\end{equation}
	where
	\[
	B_c
	=
	d(c)\bigl(I-d(c)S_c\bigr)^{-1}
	\]
	and
	\begin{equation}
	|R(c)|
	\leq
	C|\nabla_{\partial C}d(c)|^2.
	\label{eq:linearization-remainder}
	\end{equation}
	
	Because $\partial C$ and $d$ are of class $C^2$, the map
	$c\mapsto B_c$ is continuous near $c_*$. Hence
	\begin{equation}
	B_{c(u)}
	=
	B_*+o(1).
	\label{eq:B-local-continuity}
	\end{equation}
	
	Moreover,
	\[
	\nabla_{\partial C}d(c_*)=0.
	\]
	At a critical point, the differential of the tangential gradient
	is the Hessian operator $H_*$. Therefore, in the above local
	trivialization,
	\begin{equation}
	\nabla_{\partial C}d(c(u))
	=
	H_*u+o(|u|).
	\label{eq:gradient-hessian-local}
	\end{equation}
	In particular,
	\[
	|\nabla_{\partial C}d(c(u))|
	=
	O(|u|).
	\]
	It follows from \eqref{eq:linearization-remainder} that
	\begin{equation}
	R(c(u))
	=
	O(|u|^2)
	=
	o(|u|).
	\label{eq:linearization-remainder-local}
	\end{equation}
	
	Combining
	\eqref{eq:B-local-continuity} and
	\eqref{eq:gradient-hessian-local}, we obtain
	\begin{equation}
	B_{c(u)}
	\nabla_{\partial C}d(c(u))
	=
	B_*H_*u+o(|u|).
	\label{eq:preconditioned-hessian-local}
	\end{equation}
	
	Substituting
	\eqref{eq:linearization-remainder-local} and
	\eqref{eq:preconditioned-hessian-local} into
	\eqref{eq:linearization-local-expansion} gives
	\[
	F(c(u))-c(u)
	=
	B_*H_*u+o(|u|).
	\]
	Using \eqref{eq:local-chart-expansion},
	\[
	F(c(u))
	=
	c_*+
	\bigl(I+B_*H_*\bigr)u
	+
	o(|u|).
	\]
	Hence the local representation of $F$ is Fréchet differentiable
	at $u=0$, and
	\[
	DF(c_*)\big|_{T_{c_*}\partial C}
	=
	I+B_*H_*.
	\]
	This proves \eqref{eq:critical-linearization}.
\end{proof}
\begin{remark}
	\label{rem:exact-linearization}
	Although Theorem~\ref{thm:critical-linearization} is derived from
	the local expansion of Section~6, formula
	\eqref{eq:critical-linearization} is the exact derivative of the
	return map at the fixed point. No remainder term remains in the
	linearization.
\end{remark}
\subsection{Spectral Form of the Linearization}

The linearization obtained in
Theorem~\ref{thm:critical-linearization} involves the product
$B_*H_*$ of two self-adjoint operators. Although this product is
not self-adjoint in general, its spectrum has a simple symmetric
representation because $B_*$ is positive definite.

Recall that
\[
B_*
=
d_*
\bigl(I-d_*S_*\bigr)^{-1}.
\]
Since $d_*>0$ and the eigenvalues of $I-d_*S_*$ are
$1+d_*\kappa_i>0$, the operator $B_*$ is positive definite and
admits a unique positive square root $B_*^{1/2}$.

Define
\begin{equation}
K_*
:=
B_*^{1/2}H_*B_*^{1/2}.
\label{eq:symmetric-linearization-operator}
\end{equation}

\begin{proposition}[Symmetric spectral representation]
	\label{prop:symmetric-spectral-representation}
	The operator $K_*$ is self-adjoint, and
	$B_*H_*$ is similar to $K_*$. More precisely,
	\begin{equation}
	B_*H_*
	=
	B_*^{1/2}
	K_*
	B_*^{-1/2}.
	\label{eq:BH-similarity}
	\end{equation}
	Consequently, $B_*H_*$ is diagonalizable over $\mathbb R$ and all
	of its eigenvalues are real.
\end{proposition}

\begin{proof}
	Since both $B_*^{1/2}$ and $H_*$ are self-adjoint,
	\[
	K_*^*
	=
	B_*^{1/2}H_*B_*^{1/2}
	=
	K_*.
	\]
	Thus $K_*$ is self-adjoint.
	
	Moreover,
	\[
	\begin{aligned}
	B_*^{1/2}K_*B_*^{-1/2}
	&=
	B_*^{1/2}
	\bigl(B_*^{1/2}H_*B_*^{1/2}\bigr)
	B_*^{-1/2}
	\\
	&=
	B_*H_*,
	\end{aligned}
	\]
	which proves \eqref{eq:BH-similarity}. Since a self-adjoint operator
	has a real spectrum and is diagonalizable, the same is true for the
	similar operator $B_*H_*$.
\end{proof}
Let
\[
\mu_1,\ldots,\mu_{N-1}
\]
denote the eigenvalues of $B_*H_*$, counted with multiplicity.
Equivalently, they are the eigenvalues of the self-adjoint operator
$K_*$. By
\eqref{eq:critical-linearization}, the eigenvalues of the derivative
of the return map are therefore
\begin{equation}
\lambda_j
=
1+\mu_j,
\qquad
j=1,\ldots,N-1.
\label{eq:return-linearization-eigenvalues}
\end{equation}
\begin{proposition}[Inertia of the preconditioned Hessian]
	\label{prop:preconditioned-hessian-inertia}
	The operators $K_*$ and $H_*$ have the same numbers of positive,
	negative, and zero eigenvalues. Consequently, the signs of the
	eigenvalues $\mu_j$ are determined by the Morse type of the critical
	point $c_*$.
\end{proposition}

\begin{proof}
	For every $v\in T_{c_*}\partial C$,
	\[
	\langle K_*v,v\rangle
	=
	\left\langle
	H_*B_*^{1/2}v,
	B_*^{1/2}v
	\right\rangle.
	\]
	Since $B_*^{1/2}$ is invertible, the quadratic forms associated with
	$K_*$ and $H_*$ are related by an invertible change of variables.
	Sylvester's law of inertia therefore implies that they have the same
	inertia.
\end{proof}
In particular:

\begin{itemize}
	\item if $c_*$ is a strict local minimum and $H_*$ is positive
	definite, then
	\[
	\mu_j>0
	\qquad\text{for all }j;
	\]
	\item if $c_*$ is a strict local maximum and $H_*$ is negative
	definite, then
	\[
	\mu_j<0
	\qquad\text{for all }j;
	\]
	\item if $c_*$ is a nondegenerate saddle point, then the set of
	eigenvalues $\{\mu_j\}$ contains both positive and negative
	values.
\end{itemize}
The fixed point $c_*$ is linearly attracting precisely when
\begin{equation}
|\lambda_j|<1
\qquad\text{for every }j,
\label{eq:linear-attraction-condition}
\end{equation}
or equivalently,
\begin{equation}
-2<\mu_j<0
\qquad\text{for every }j.
\label{eq:mu-attraction-condition}
\end{equation}

Thus negative definiteness of the Hessian is necessary for linear
attraction at a nondegenerate critical point, but it is not by
itself sufficient: the negative eigenvalues of the preconditioned
Hessian must also remain above the threshold $-2$.
\subsection{Local Stability of Nondegenerate Critical Points}

We now translate the spectral information obtained above into a
local stability classification of nondegenerate fixed points.

\begin{theorem}[Local stability classification]
	\label{thm:local-stability-classification}
	Let $c_*\in\partial C$ be a nondegenerate critical point of $d$.
	Assume that $\partial C$ and $d$ are of class $C^2$ in a
	neighbourhood of $c_*$. Let
	\[
	\mu_1,\ldots,\mu_{N-1}
	\]
	be the eigenvalues of the preconditioned Hessian
	\[
	B_*H_*,
	\qquad
	B_*
	=
	d_*
	\bigl(I-d_*S_*\bigr)^{-1}.
	\]
	Then the eigenvalues of the linearized return map are
	\[
	\lambda_j=1+\mu_j.
	\]
	
	The following alternatives hold.
	
	\begin{enumerate}
		\item If $c_*$ is a strict local minimum of $d$, then
		\[
		\mu_j>0
		\qquad\text{for all }j,
		\]
		and hence
		\[
		\lambda_j>1
		\qquad\text{for all }j.
		\]
		Therefore $c_*$ is a hyperbolic repelling fixed point.
		
		\item If $c_*$ is a strict local maximum of $d$, then
		\[
		\mu_j<0
		\qquad\text{for all }j.
		\]
		The fixed point is linearly hyperbolic attracting if and only if
		\begin{equation}
		-2<\mu_j<0
		\qquad\text{for all }j.
		\label{eq:maximum-attraction-condition}
		\end{equation}
		Equivalently,
		\[
		|\lambda_j|<1
		\qquad\text{for all }j.
		\]
		
		\item If $c_*$ is a nondegenerate saddle point of $d$, then
		some $\mu_j$ are positive and some are negative. In particular,
		at least one eigenvalue of $DF(c_*)$ satisfies
		\[
		\lambda_j>1.
		\]
		Hence $c_*$ is unstable.
	\end{enumerate}
\end{theorem}

\begin{proof}
	The spectral relation
	\[
	\lambda_j=1+\mu_j
	\]
	follows from
	\eqref{eq:return-linearization-eigenvalues}.
	
	If $c_*$ is a strict local minimum, the Hessian $H_*$ is positive
	definite. Proposition~\ref{prop:preconditioned-hessian-inertia}
	therefore gives $\mu_j>0$ for every $j$, and hence
	$\lambda_j>1$. Thus the fixed point is repelling.
	
	If $c_*$ is a strict local maximum, $H_*$ is negative definite, so
	$\mu_j<0$ for every $j$. Linear attraction is equivalent to
	\[
	|1+\mu_j|<1
	\]
	for every $j$, which is precisely
	\[
	-2<\mu_j<0.
	\]
	
	Finally, if $c_*$ is a nondegenerate saddle point, then $H_*$ has
	both positive and negative eigenvalues. By
	Proposition~\ref{prop:preconditioned-hessian-inertia}, the same is
	true of the set $\{\mu_j\}$. Hence at least one $\mu_j$ is positive,
	so the corresponding eigenvalue
	\[
	\lambda_j=1+\mu_j
	\]
	satisfies $\lambda_j>1$. Therefore the fixed point is unstable.
\end{proof}
\begin{remark}
	\label{rem:maximum-stability-threshold}
	A strict local maximum is not automatically attracting. The
	negative eigenvalues of the preconditioned Hessian must satisfy the
	quantitative bound
	\[
	\mu_j>-2.
	\]
	If some $\mu_j=-2$, then the corresponding multiplier equals $-1$
	and the linearization is nonhyperbolic. If some $\mu_j<-2$, then
	the corresponding multiplier satisfies
	\[
	\lambda_j<-1,
	\]
	and the fixed point is unstable in that direction.
\end{remark}
\begin{corollary}
	\label{cor:stability-reversal}
	Under the assumptions of
	Theorem~\ref{thm:local-stability-classification}, nondegenerate
	local minima of the thickness function are repelling fixed points
	of the return map. Nondegenerate local maxima are the only critical
	points that can be locally attracting, subject to the spectral
	condition \eqref{eq:maximum-attraction-condition}.
\end{corollary}
\section{Numerical Validation}

The purpose of this section is to provide direct numerical checks
of the  return-map formulas derived in Sections~6 and~7.
We restrict the computations to the unit-circle core, for which the
return step can be evaluated explicitly. This allows the numerical
iteration to be performed from the exact geometric equations rather
than from the first-order approximation.

The examples below are intended only to validate the
first-order displacement and the local multiplier at a fixed point.
No global convergence or Lyapunov conclusion is inferred from these
computations.

\subsection{Exact Numerical Return Map on the Unit Circle}

Let
\[
C=\overline{B_1(0)}\subset\mathbb R^2,
\qquad
c(\theta)=(\cos\theta,\sin\theta),
\]
and let $d(\theta)>0$ be a smooth $2\pi$-periodic thickness
function. Denote by
\[
e_\theta=(-\sin\theta,\cos\theta)
\]
the positively oriented unit tangent.

As derived in Section~6.4, define
\begin{equation}
p(\theta)
=
\frac{d'(\theta)}{1+d(\theta)}.
\label{eq:num-p}
\end{equation}
The exact inward normal at the corresponding point of the outer
boundary is
\begin{equation}
n(\Phi(c(\theta)))
=
\frac{-c(\theta)+p(\theta)e_\theta}
{\sqrt{1+p(\theta)^2}}.
\label{eq:num-inward-normal}
\end{equation}

Introducing
\[
\sigma
=
\frac{t}{\sqrt{1+p^2}},
\]
the condition that the return point belongs to the unit circle gives
the exact closure equation
\begin{equation}
(1+d-\sigma)^2+\sigma^2p^2=1.
\label{eq:num-closure}
\end{equation}
The branch corresponding to the first intersection of the inward
normal ray with the convex core is
\begin{equation}
\sigma(\theta)
=
\frac{
	1+d(\theta)
	-
	\sqrt{
		1-d(\theta)\bigl(2+d(\theta)\bigr)p(\theta)^2
	}
}{
	1+p(\theta)^2
}.
\label{eq:num-sigma}
\end{equation}

For numerical purposes, set
\begin{equation}
a(\theta)
=
1+d(\theta)-\sigma(\theta),
\qquad
b(\theta)
=
\sigma(\theta)p(\theta).
\label{eq:num-ab}
\end{equation}
Then
\begin{equation}
F(c(\theta))
=
a(\theta)c(\theta)
+
b(\theta)e_\theta.
\label{eq:num-return-point}
\end{equation}
The closure equation implies
\[
a(\theta)^2+b(\theta)^2=1.
\]
Consequently, the exact angular displacement is obtained from
\begin{equation}
\delta(\theta)
=
\atanTwo\bigl(b(\theta),a(\theta)\bigr).
\label{eq:num-delta}
\end{equation}
and one iteration of the return map is
\begin{equation}
\theta_{k+1}
=
\theta_k+\delta(\theta_k)
\pmod{2\pi}.
\label{eq:num-iteration}
\end{equation}

Thus each numerical step is computed through the sequence
\[
\begin{aligned}
\theta_k
&\longmapsto
\bigl(d(\theta_k),d'(\theta_k)\bigr)
\longmapsto
p(\theta_k)
\longmapsto
\sigma(\theta_k)
\longmapsto
\delta(\theta_k)\\
&\longmapsto
\theta_{k+1}.
\end{aligned}
\]

At every step we additionally monitor the radicand
\begin{equation}
\mathcal R(\theta)
=
1-
d(\theta)\bigl(2+d(\theta)\bigr)p(\theta)^2.
\label{eq:num-radicand}
\end{equation}
For both thickness profiles used below, $d(\theta)>0$ and a direct
substitution shows that
\[
\mathcal R(\theta)>0
\qquad
\text{for every }\theta\in[0,2\pi].
\]
Hence the first-intersection branch
\eqref{eq:num-sigma} is real on the whole circle for both
numerical examples.

As a basic consistency check, if $d'(\theta_*)=0$, then
$p(\theta_*)=0$ and
\[
\sigma(\theta_*)=d(\theta_*),
\qquad
\delta(\theta_*)=0.
\]
Hence
\[
F(c(\theta_*))=c(\theta_*),
\]
in agreement with the fixed-point characterization of
Theorem~\ref{thm:fixed-critical-equivalence}.
\subsection{Validation of the First-Order Expansion}

We first compare the exact unit-circle return map with the
first-order angular displacement derived in Section~6.4.

Consider the thickness profile
\begin{equation}
d(\theta)
=
0.5+0.2\cos(2\theta),
\label{eq:num-example1-thickness}
\end{equation}
so that
\begin{equation}
d'(\theta)
=
-0.4\sin(2\theta).
\label{eq:num-example1-derivative}
\end{equation}
We choose the initial condition
\begin{equation}
\theta_0=1.
\label{eq:num-example1-initial}
\end{equation}

For each iterate, the exact angular displacement
$\delta_{\mathrm{ex}}(\theta)$ is evaluated from
\eqref{eq:num-delta}, whereas the first-order prediction is
\begin{equation}
\delta_{\mathrm{lin}}(\theta)
=
\frac{d(\theta)}
{1+d(\theta)}
d'(\theta).
\label{eq:num-linear-prediction}
\end{equation}

The first iterates of the exact return map are
\begin{align*}
1
&\longmapsto
0.89127
\longmapsto
0.76600
\longmapsto
0.62846
\longmapsto
0.48879\\
&\longmapsto
0.36087
\longmapsto
0.25567
\longmapsto
0.17638
\longmapsto\cdots .
\end{align*}
Table~\ref{tab:num-first-order-validation} compares the exact
angular displacement with the first-order approximation.

\begin{table*}[t]
	\centering
	\caption{Exact and first-order angular displacements for
		$d(\theta)=0.5+0.2\cos(2\theta)$ with $\theta_0=1$.}
	\label{tab:num-first-order-validation}
	\begin{tabular}{c c c c c c}
		\hline
		$k$
		&
		$\theta_k$
		&
		$d'(\theta_k)$
		&
		$\delta_{\mathrm{ex}}$
		&
		$\delta_{\mathrm{lin}}$
		&
		$|\delta_{\mathrm{ex}}-\delta_{\mathrm{lin}}|$
		\\
		\hline
		0 & 1.000000 & -0.363719 & -0.108725 & -0.106995 & 0.001730 \\
		1 & 0.891275 & -0.391066 & -0.125268 & -0.122839 & 0.002429 \\
		2 & 0.766007 & -0.399699 & -0.137540 & -0.134604 & 0.002937 \\
		3 & 0.628467 & -0.380459 & -0.139674 & -0.136848 & 0.002826 \\
		4 & 0.488793 & -0.331660 & -0.127921 & -0.125891 & 0.002030 \\
		5 & 0.360872 & -0.264278 & -0.105202 & -0.104122 & 0.001079 \\
		6 & 0.255671 & -0.195739 & -0.079290 & -0.078839 & 0.000450 \\
		7 & 0.176381 & -0.138197 & -0.056472 & -0.056311 & 0.000161 \\
		\hline
	\end{tabular}
\end{table*}

The exact and first-order displacements have the same sign
throughout the displayed trajectory. In particular, since
$d'(\theta_k)<0$, the return map moves toward decreasing values of
$\theta$, in agreement with the positive coefficient in
\eqref{eq:num-linear-prediction}.

Moreover, as the orbit approaches the critical point at
$\theta=0$, the difference between the exact and first-order
displacements decreases rapidly. This is consistent with the
estimate
\begin{equation}
\delta_{\mathrm{ex}}(\theta)
=
\frac{d(\theta)}
{1+d(\theta)}
d'(\theta)
+
O\!\left(|d'(\theta)|^2\right)
\label{eq:num-first-order-consistency}
\end{equation}
derived in Section~6.4.

For the displayed iterates, the radicand
$\mathcal R(\theta_k)$ in \eqref{eq:num-radicand} remains positive;
its smallest value is approximately
\[
0.910516.
\]
Thus all reported steps belong to the real first-intersection branch
of the exact return construction.
\begin{figure}[ht]
	\centering
	\begin{tikzpicture}
	\begin{axis}[
	width=0.82\columnwidth,
	height=0.50\columnwidth,
	xlabel={$k$},
	ylabel={$\theta_k$},
	xmin=0, xmax=7,
	ymin=0, ymax=1.05,
	xtick={0,1,2,3,4,5,6,7},
	grid=major
	]
	\addplot[
	mark=*,
	] coordinates {
		(0,1.000000)
		(1,0.891275)
		(2,0.766007)
		(3,0.628467)
		(4,0.488793)
		(5,0.360872)
		(6,0.255671)
		(7,0.176381)
	};
	\addplot[
	dashed,
	domain=0:7,
	samples=2
	] {0};
	\end{axis}
	\end{tikzpicture}
	\caption{Exact iterates of the return map for
		$d(\theta)=0.5+0.2\cos(2\theta)$ with initial condition
		$\theta_0=1$. The dashed line marks the critical point
		$\theta=0$.}
	\label{fig:num-example1-trajectory}
\end{figure}

\subsection{Validation of the Local Multiplier}

We next test the local linearization derived in Section~7 using
the thickness profile
\begin{equation}
d(\theta)
=
0.6+0.4\cos(3\theta).
\label{eq:num-example2-thickness}
\end{equation}
Its first two derivatives are
\begin{equation}
d'(\theta)
=
-1.2\sin(3\theta),
\qquad
d''(\theta)
=
-3.6\cos(3\theta).
\label{eq:num-example2-derivatives}
\end{equation}

The points
\[
0,\qquad
\frac{2\pi}{3},
\qquad
\frac{4\pi}{3}
\]
are local maxima of the thickness function, whereas
\[
\frac{\pi}{3},\qquad
\pi,\qquad
\frac{5\pi}{3}
\]
are local minima.

We consider the maximum
\[
\theta_*=0.
\]
At this point,
\[
d(0)=1,
\qquad
d''(0)=-3.6.
\]
For the unit circle, the tangential preconditioning factor is
\[
B_*=\frac{d(0)}{1+d(0)}=\frac12.
\]
Hence the one-dimensional version of
Theorem~\ref{thm:critical-linearization} gives
\begin{equation}
F'(\theta_*)
=
1+
\frac{d(\theta_*)}{1+d(\theta_*)}
d''(\theta_*).
\label{eq:num-multiplier-formula}
\end{equation}
Therefore,
\begin{equation}
F'(0)
=
1+\frac12(-3.6)
=
-0.8.
\label{eq:num-multiplier-value}
\end{equation}

The negative multiplier predicts that sufficiently small
perturbations of the fixed point change sign after each iteration,
while their magnitude is asymptotically reduced by the factor
$0.8$.

To test this prediction, we apply the exact return-map algorithm of
Section~8.1 with initial condition
\begin{equation}
\theta_0=0.5.
\label{eq:num-example2-initial}
\end{equation}
The resulting orbit begins
\begin{align*}
&0.500000
\longmapsto
-0.169402
\longmapsto
0.133200
\longmapsto
-0.105476\\
&\longmapsto
0.083855
\longmapsto
-0.066823
\longmapsto
0.053328
\longmapsto\cdots .
\end{align*}
For $\theta_k\neq0$, define the observed local ratio
\begin{equation}
q_k
:=
\frac{\theta_{k+1}-\theta_*}
{\theta_k-\theta_*}
=
\frac{\theta_{k+1}}{\theta_k}.
\label{eq:num-observed-multiplier}
\end{equation}
Table~\ref{tab:num-multiplier-validation} shows that the observed
ratios approach the theoretical multiplier $-0.8$.
\begin{table}[t]
	\centering
	\small
	\caption{Validation of the local multiplier at the fixed point
		$\theta_*=0$ for
		$d(\theta)=0.6+0.4\cos(3\theta)$.}
	\label{tab:num-multiplier-validation}
	\begin{tabular}{c c c c}
		\hline
		$k$ & $\theta_k$ & $\theta_{k+1}$ & $q_k$ \\
		\hline
		0 &  0.500000 & -0.169402 & -0.338804 \\
		1 & -0.169402 &  0.133200 & -0.786295 \\
		2 &  0.133200 & -0.105476 & -0.791861 \\
		3 & -0.105476 &  0.083855 & -0.795013 \\
		4 &  0.083855 & -0.066823 & -0.796891 \\
		5 & -0.066823 &  0.053328 & -0.798043 \\
		6 &  0.053328 & -0.042596 & -0.798760 \\
		7 & -0.042596 &  0.034043 & -0.799212 \\
		8 &  0.034043 & -0.027217 & -0.799497 \\
		9 & -0.027217 &  0.021765 & -0.799679 \\
		10 & 0.021765 & -0.017408 & -0.799795 \\
		\hline
	\end{tabular}
\end{table}

In particular,
the values displayed in the last rows are already close to the
linear prediction
\[
\theta_{k+1}-\theta_*
\approx
-0.8\,(\theta_k-\theta_*).
\]

The alternating signs of the computed iterates are consistent with
the local behavior of an attracting fixed point having a negative
multiplier; they should not be interpreted as evidence of a
period-two orbit. Over the displayed trajectory, the deviations
from $\theta_*=0$ decrease in magnitude while alternating in sign,
in agreement with the exact local multiplier
\eqref{eq:num-multiplier-value}.

\begin{figure}[ht]
	\centering
	\begin{tikzpicture}
	\begin{axis}[
	width=0.82\columnwidth,
	height=0.50\columnwidth,
	xlabel={$k$},
	ylabel={$\theta_k-\theta_*$},
	xmin=0,
	xmax=10,
	ymin=-0.20,
	ymax=0.52,
	xtick={0,1,2,3,4,5,6,7,8,9,10},
	grid=major,
	axis lines=left
	]
	\addplot[
	mark=*,
	] coordinates {
		(0, 0.500000)
		(1,-0.169402)
		(2, 0.133200)
		(3,-0.105476)
		(4, 0.083855)
		(5,-0.066823)
		(6, 0.053328)
		(7,-0.042596)
		(8, 0.034043)
		(9,-0.027217)
		(10,0.021765)
	};
	
	\addplot[
	dashed,
	domain=0:10,
	samples=2
	] {0};
	\end{axis}
	\end{tikzpicture}
	\caption{Signed deviation from the fixed point $\theta_*=0$ for
		$d(\theta)=0.6+0.4\cos(3\theta)$. The alternating sign and decreasing
		magnitude are consistent with the theoretical multiplier
		$F'(0)=-0.8$.}
	\label{fig:num-example2-multiplier}
\end{figure}
\subsection{Reproducibility Details}

All numerical values reported in this section were obtained by
direct evaluation of the exact unit-circle return map
\eqref{eq:num-p}--\eqref{eq:num-iteration}. No numerical
root-finding procedure and no first-order approximation were used
to generate the iterates.

The computations were performed in double-precision floating-point
arithmetic. At each iteration step, the quantities
\[
d(\theta_k),\;
d'(\theta_k),\;
p(\theta_k),\;
\mathcal R(\theta_k),\;
\sigma(\theta_k),
\;
\delta(\theta_k)
\]
were evaluated successively from the formulas given in
Section~8.1.

Angles were represented in radians. Whenever necessary, the
resulting angle was reduced modulo $2\pi$. For computations near
the fixed point $\theta_*=0$, the locally continuous
representative of the angular coordinate was used so that signed
deviations from the fixed point remained meaningful.

The square-root branch in \eqref{eq:num-sigma} was accepted only
when
\begin{equation}
\mathcal R(\theta_k)
=
1-
d(\theta_k)
\bigl(2+d(\theta_k)\bigr)
p(\theta_k)^2
\geq 0.
\label{eq:num-reproducibility-radicand}
\end{equation}
For all iterates reported in Tables
\ref{tab:num-first-order-validation} and
\ref{tab:num-multiplier-validation}, this condition was satisfied.

For an iteration continued beyond the values displayed in the
tables, the stopping criterion was
\begin{equation}
|\theta_{k+1}-\theta_k|<10^{-12}.
\label{eq:num-stopping-criterion}
\end{equation}
The values printed in the tables are rounded to six decimal
places, whereas all intermediate computations were performed
without such rounding.

For the first example,
\[
d(\theta)=0.5+0.2\cos(2\theta),
\qquad
\theta_0=1,
\]
the smallest value of the radicand among the displayed iterates is
approximately
\[
0.910516.
\]
For the second example,
\[
d(\theta)=0.6+0.4\cos(3\theta),
\qquad
\theta_0=0.5,
\]
the first displayed step gives the smallest radicand,
approximately
\[
0.107608.
\]
Over the subsequent computed iterates, the radicand increases
toward $1$ as the deviations from $\theta_*=0$ decrease.

The numerical data therefore arise directly from the exact
geometric return construction and can be reproduced independently
from the formulas stated above.

\section{Regularity and Bilipschitz Structure}

The radial representation
\[
\Phi(c)=c+d(c)\nu(c)
\]
encodes the outer boundary entirely in terms of the convex core and
the thickness function. In the preceding sections, only the
regularity required for the return-map construction was imposed.
We now record two basic structural consequences of this
parametrization: the correspondence between the regularity of
$d$ and that of the outer boundary, and quantitative bilipschitz
bounds for $\Phi$.

The discussion is deliberately restricted to the regularity
properties needed for the geometric framework developed in this
paper.
\subsection{Regularity of the Thickness and the Outer Boundary}

Under the standing assumptions of Section~2, the thickness
function is locally Lipschitz by Proposition~\ref{prop:thickness-local-lipschitz}.
We now record the corresponding higher-regularity statement under
additional smoothness assumptions.

\begin{proposition}[Regularity transfer]
	\label{prop:regularity-transfer}
	Let $m\geq1$ and $0\leq\alpha\leq1$. Assume that
	$\partial C$ is of class $C^{m+1,\alpha}$.
	
	\begin{enumerate}
		\item If
		\[
		d\in C^{m,\alpha}(\partial C),
		\]
		then
		\[
		\Phi(c)=c+d(c)\nu(c)
		\]
		is of class $C^{m,\alpha}$. Moreover, $\Phi$ is a
		$C^{m,\alpha}$ parametrization of
		$\partial\Omega\setminus C$.
		
		\item Conversely, let $c_0\in\partial C$ and
		$x_0=\Phi(c_0)$. If $\partial\Omega$ is of class
		$C^{m,\alpha}$ near $x_0$ and the radial intersection is
		transverse,
		\[
		\langle n(x_0),\nu(c_0)\rangle\neq0,
		\]
		then $d$ is of class $C^{m,\alpha}$ in a neighbourhood
		of $c_0$.
	\end{enumerate}
\end{proposition}
\begin{proof}
	For the first assertion, the regularity of $\partial C$ implies
	that the outward normal field
	\[
	\nu:\partial C\longrightarrow\mathbb S^{N-1}
	\]
	is of class $C^{m,\alpha}$. Hence
	\[
	\Phi(c)=c+d(c)\nu(c)
	\]
	is $C^{m,\alpha}$ whenever $d$ is $C^{m,\alpha}$.
	
	It remains to verify that the parametrization is nondegenerate.
	For $v\in T_c\partial C$, equation~\eqref{eq:differential-radial-map}
	gives
	\[
	D\Phi(c)[v]
	=
	\bigl(I-d(c)S_c\bigr)v
	+
	\langle\nabla_{\partial C}d(c),v\rangle\nu(c).
	\]
	The tangential and normal components are orthogonal. Therefore,
	if
	\[
	D\Phi(c)[v]=0,
	\]
	then necessarily
	\[
	\bigl(I-d(c)S_c\bigr)v=0.
	\]
	For a convex core, the eigenvalues of
	$I-d(c)S_c$ are
	\[
	1+d(c)\kappa_i(c)>0,
	\]
	so the operator is invertible and hence $v=0$.
	Thus $D\Phi(c)$ has full rank.
	
	By Proposition~\ref{prop:radial-accessibility}, $\Phi$ is already
	globally bijective from $\partial C$ onto
	$\partial\Omega\setminus C$. The local inverse-function theorem
	therefore shows that $\Phi$ provides a $C^{m,\alpha}$
	parametrization of the outer boundary.
	
	For the converse, choose a local $C^{m,\alpha}$ defining function
	$\rho$ for $\partial\Omega$ near $x_0$, so that
	\[
	\partial\Omega=\{\rho=0\},
	\qquad
	\nabla\rho(x_0)\neq0.
	\]
	Let $c=c(y)$ be a local parametrization of $\partial C$ near
	$c_0$ and define
	\begin{equation}
	G(y,t)
	=
	\rho\bigl(c(y)+t\nu(c(y))\bigr).
	\label{eq:regularity-implicit-G}
	\end{equation}
	By construction,
	\[
	G(y,d(c(y)))=0.
	\]
	At the point corresponding to $c_0$,
	\begin{align}
	\partial_tG(y_0,d(c_0))
	&=
	\left\langle
	\nabla\rho(x_0),\nu(c_0)
	\right\rangle .
	\label{eq:regularity-transversality}
	\end{align}
	Since $\nabla\rho(x_0)$ is normal to $\partial\Omega$ at $x_0$,
	the transversality assumption implies
	\[
	\partial_tG(y_0,d(c_0))\neq0.
	\]
	The classical implicit-function theorem therefore yields a unique
	$C^{m,\alpha}$ solution
	\[
	t=d(c(y))
	\]
	near $y_0$. Hence $d$ is $C^{m,\alpha}$ near $c_0$.
\end{proof}
\begin{remark}
	\label{rem:regularity-transfer-scope}
	The converse implication uses the radial transversality explicitly.
	It does not rely on assuming in advance that the inverse of the
	radial map is smooth. This avoids the circular inverse-map argument
	that can arise if regularity of $\Phi^{-1}$ is used before
	nondegeneracy of $\Phi$ has been established.
\end{remark}
\subsection{Bilipschitz Bounds for the Radial Map}

The radial parametrization admits global quantitative bounds that
follow directly from convex projection. In particular, no smallness
condition involving the thickness or the curvature of the convex
core is required.

Let
\[
L_d:=\operatorname{Lip}(d),
\qquad
L_\nu:=\operatorname{Lip}(\nu),
\qquad
d_{\max}:=\|d\|_{L^\infty(\partial C)}.
\]

Under the standing assumptions, these quantities are finite:
$\nu$ is Lipschitz because $\partial C$ is of class $C^{1,1}$,
while Proposition~\ref{prop:thickness-local-lipschitz}, together with
compactness of $\partial C$, gives a global Lipschitz constant for
$d$.

\begin{proposition}[Global bilipschitz estimate]
	\label{prop:radial-bilipschitz}
	Let $\Omega\in\mathcal O_C$ and let
	\[
	\Phi(c)=c+d(c)\nu(c).
	\]
	Then for every $c_1,c_2\in\partial C$,
	\begin{equation}
	|c_1-c_2|
	\leq
	|\Phi(c_1)-\Phi(c_2)|
	\leq
	\bigl(
	1+L_d+d_{\max}L_\nu
	\bigr)
	|c_1-c_2|.
	\label{eq:radial-bilipschitz}
	\end{equation}
	Consequently,
	\[
	\Phi:\partial C\longrightarrow\partial\Omega\setminus C
	\]
	is bilipschitz. Moreover, its inverse satisfies
	\begin{equation}
	\operatorname{Lip}(\Phi^{-1})\leq1.
	\label{eq:inverse-radial-lipschitz}
	\end{equation}
\end{proposition}

\begin{proof}
	By Proposition~\ref{prop:radial-accessibility},
	\[
	\Phi^{-1}
	=
	P_C\big|_{\partial\Omega\setminus C}.
	\]
	Since the metric projection onto a closed convex set is
	non-expansive,
	\[
	|P_C(x_1)-P_C(x_2)|
	\leq
	|x_1-x_2|
	\]
	for all $x_1,x_2\in\mathbb R^N$.
	
	Taking
	\[
	x_i=\Phi(c_i),
	\qquad i=1,2,
	\]
	and using
	\[
	P_C(\Phi(c_i))=c_i,
	\]
	gives
	\begin{equation}
	|c_1-c_2|
	\leq
	|\Phi(c_1)-\Phi(c_2)|.
	\label{eq:radial-lower-bound}
	\end{equation}
	
	For the upper bound, write
	\begin{align}
	\Phi(c_1)-\Phi(c_2)
	&=
	c_1-c_2
	+
	d(c_1)\nu(c_1)
	-
	d(c_2)\nu(c_2)
	\nonumber\\
	&=
	c_1-c_2
	+
	\bigl(d(c_1)-d(c_2)\bigr)\nu(c_1)
	\nonumber\\
	&\qquad
	+
	d(c_2)\bigl(\nu(c_1)-\nu(c_2)\bigr).
	\label{eq:radial-upper-decomposition}
	\end{align}
	Hence
	\begin{align}
	|\Phi(c_1)-\Phi(c_2)|
	&\leq
	|c_1-c_2|
	+
	|d(c_1)-d(c_2)|
	\nonumber\\
	&\qquad
	+
	d(c_2)|\nu(c_1)-\nu(c_2)|
	\nonumber\\
	&\leq
	\bigl(
	1+L_d+d_{\max}L_\nu
	\bigr)
	|c_1-c_2|.
	\label{eq:radial-upper-bound}
	\end{align}
	Combining
	\eqref{eq:radial-lower-bound}
	and
	\eqref{eq:radial-upper-bound}
	proves
	\eqref{eq:radial-bilipschitz}.
	
	Finally, the lower bound is equivalent to
	\[
	|\Phi^{-1}(x_1)-\Phi^{-1}(x_2)|
	\leq
	|x_1-x_2|,
	\]
	which proves
	\eqref{eq:inverse-radial-lipschitz}.
\end{proof}
\begin{remark}
	\label{rem:no-bilipschitz-smallness}
	The estimate
	\eqref{eq:radial-bilipschitz} does not require a condition of the
	form
	\[
	L_d+d_{\max}L_\nu<1.
	\]
	The lower Lipschitz bound follows instead from the convexity of $C$
	and the non-expansiveness of the metric projection. Thus the
	bilipschitz property is a structural consequence of the global
	radial representation established in Proposition
	\ref{prop:radial-accessibility}, rather than a perturbative
	smallness result.
\end{remark}
\section{Discussion and Outlook}

The construction developed in this paper shows that the geometric
normal properties of the class $\mathcal O_C$ naturally generate a
discrete dynamical system on the boundary of the convex core.
The essential mechanism is the geometric round-trip
\[
c
\longmapsto
\Phi(c)
=
c+d(c)\nu(c)
\longmapsto
F(c),
\]
where the outward radial displacement is determined by the normal
geometry of $\partial C$, whereas the return step is determined by
the inward normal geometry of $\partial\Omega$.

The analysis reveals that these two normal structures are coupled
through the thickness function. In particular, the first-order
return displacement is governed by the positive tangential operator
\[
B_c
=
d(c)\bigl(I-d(c)S_c\bigr)^{-1},
\]
so that, near a critical point of the thickness,
\[
F(c)-c
=
B_c\nabla_{\partial C}d(c)
+
O\!\left(
|\nabla_{\partial C}d(c)|^2
\right).
\]
Thus curvature of the convex core enters directly into the
leading-order dynamics through an anisotropic preconditioning of
the thickness gradient.

This geometric structure also clarifies the local equilibrium
theory. At differentiability points, critical points of the thickness
are precisely the fixed points of the return map. Under the
additional smoothness assumptions used in Section~7, this
equivalence holds pointwise, and the derivative at a critical
point is
\[
DF(c_*)
=
I+B_*H_*,
\]
where
\[
B_*
=
d_*
\bigl(I-d_*S_*\bigr)^{-1}
\]
and $H_*$ is the Hessian operator of the thickness at $c_*$. The
local stability of a fixed point therefore depends on the
interaction between the second-order variation of the thickness
and the curvature of the convex core, rather than on the Hessian
of $d$ alone.

The numerical examples of Section~8 provide direct checks of this
local theory using the exact return geometry on the unit circle.
They confirm both the  first-order displacement and the
local multiplier predicted by the linearization. Their purpose is
validation of the geometric formulas rather than a classification
of the global dynamics.

The regularity results of Section~9 complement this dynamical
picture. The radial parametrization transfers regularity between
the thickness function and the outer boundary, while convex metric
projection gives a global bilipschitz control of the radial map.
In particular, this control does not require a perturbative
smallness condition.

The present paper is intentionally limited to these foundational
properties. Several questions arising naturally from the return
construction require a separate analysis. These include the global
behaviour of the iterates, possible monotonicity or Lyapunov
structures, the existence and stability of nontrivial periodic
orbits, and the relation between the critical-point structure of
the thickness and the topology of $\partial C$. A continuous-time
limit of the return dynamics would likewise require an appropriate
scaling and a separate consistency analysis.

The framework developed here provides the geometric starting point
for such investigations. In particular, the operator
\[
d(c)\bigl(I-d(c)S_c\bigr)^{-1}
\]
identifies the interaction between thickness and curvature that
must be retained in any subsequent dynamical or topological
analysis of the return map.

\section{Conclusion}

We have developed the geometric and analytical foundations of the
return map associated with domains in the class $\mathcal O_C$.
Starting from the radial representation
\[
\Phi(c)=c+d(c)\nu(c),
\]
and the reciprocal map defined by inward normal rays from the outer
boundary, we obtained a discrete return dynamics
\[
F=\pi\circ\Phi:\partial C\longrightarrow\partial C.
\]

The exact normal geometry of the radial parametrization leads to
the local quadratic-remainder formula
\[
F(c)-c
=
d(c)\bigl(I-d(c)S_c\bigr)^{-1}
\nabla_{\partial C}d(c)
+
O\!\left(
|\nabla_{\partial C}d(c)|^2
\right).
\]
Thus the leading-order return displacement is a
curvature-preconditioned gradient step whose metric dependence is
determined by the geometry of the convex core.

We showed that, at differentiability points, the fixed points
of the return map coincide with the critical points of the
thickness function. Under the corresponding smoothness
assumptions, this equivalence holds pointwise.
The derivative  at a critical point $c_*$ is
\[
DF(c_*)
=
I+
d_*
\bigl(I-d_*S_*\bigr)^{-1}
H_*,
\]
which yields a local stability criterion through the interaction
between the Hessian of the thickness and the curvature of
$\partial C$.

The unit-circle computations provide reproducible checks of the
first-order displacement and of the local multiplier
predicted by the linearization. In addition, the radial
parametrization gives a direct regularity correspondence between
the thickness function and the outer boundary and satisfies global
bilipschitz bounds derived from convex metric projection.

These results establish the return map as a geometrically defined
discrete dynamical system whose local behaviour is controlled by
the coupled effects of thickness and curvature. They provide the
foundational framework for subsequent investigations of global
dynamics, periodic behaviour, convergence mechanisms, and the
topological structure of the thickness landscape.

\end{document}